\documentclass[11pt,a4paper,final]{article}

\usepackage{amsthm}      
\usepackage{latexsym}    
\usepackage{amssymb,amsmath} 
\usepackage{graphicx}    
\usepackage{fullpage}    

\newtheorem{theorem} {Theorem}
\newtheorem{lemma} [theorem] {Lemma}

\newtheorem{corollary} [theorem] {Corollary}
\newtheorem{observation} [theorem] {Observation}

\newtheorem{conjecture} {Conjecture}

\newtheorem{problem}[theorem]{Problem}

\DeclareMathOperator{\Conv}{Conv}
\DeclareMathOperator{\CH}{CH}

\graphicspath{{figures/}} 

\begin{document}


\title{Empty Monochromatic Simplices}

\author{Oswin Aichholzer\thanks{Institute for Software Technology,
    University of Technology, Graz, Austria, {\tt oaich@ist.tugraz.at}} \and
  Ruy Fabila-Monroy\thanks{Departamento de Matem{\'a}ticas, 
    Cinvestav, D.F. M\'exico, M\'exico, {\tt ruyfabila@math.cinsvestav.edu.mx}} \and
  Thomas Hackl\thanks{Institute for Software Technology,
    University of Technology, Graz, Austria, {\tt
      thackl@ist.tugraz.at}, Inffeldgasse 16b/II, 8010
    Graz, Austria; Tel: +43 316 873 5702; FAX: +43 316 873 5706} \and
  Clemens Huemer\thanks{Departament de Matem{\`a}tica Aplicada IV, 
    Universitat Polit{\`e}cnica de Catalunya, Barcelona, Spain, {\tt clemens.huemer@upc.edu}} \and
  Jorge Urrutia\thanks{Instituto de Matem{\'a}ticas,
    Universidad Nacional Aut{\'o}noma de M{\'e}xico,
    D.F. M\'exico, M\'exico, {\tt urrutia@matem.unam.mx}}}

\date{\today}
\maketitle

\begin{abstract}
  Let $S$ be a $k$-colored (finite) set of $n$ points in
  $\mathbb{R}^d$, $d\geq 3$, in general position, that is, no
  \mbox{$(d\!+\!1)$} points of $S$ lie in a common
  \mbox{$(d\!-\!1)$}-dimensional hyperplane. We count the number of
  empty monochromatic $d$-simplices determined by $S$, that is,
  simplices which have only points from one color class of $S$ as
  vertices and no points of $S$ in their interior. For $3 \leq k \leq
  d$ we provide a lower bound of $\Omega(n^{d-k+1+2^{-d}})$ and
  strengthen this to $\Omega(n^{d-2/3})$ for $k=2$.

  On the way we provide various results on triangulations of point
  sets in~$\mathbb{R}^d$. In particular, for any constant dimension
  $d\geq3$, we prove that every set of $n$ points ($n$ sufficiently
  large), in general position in $\mathbb{R}^d$, admits a
  triangulation with at least $dn+\Omega(\log n)$ simplices.
\end{abstract}

\section{Introduction}\label{sec:intro}

Let $S$ be a finite set of $n$ points in $\mathbb{R}^d$. Throughout
this paper we assume that $S$ is in general position, that is, no
\mbox{$(d\!+\!1)$} points of $S$ lie in a common
\mbox{$(d\!-\!1)$}-dimensional hyperplane. A more formal definition of
``general position'' can be found in Section~\ref{sec:simplcompl}. A
subset $S'$ of $S$ is said to be empty if $\Conv(S') \cap S=S'$, where
$\Conv(S')$ denotes the convex hull of $S'$ (please see
Section~\ref{sec:simplcompl} for a detailed definition). A
\emph{$k$-coloring} of $S$ is a partition of $S$ into $k$ non-empty
sets called \emph{color classes}. A subset of $S$ is said to be
\emph{monochromatic} if all its elements belong to the same color
class. A $d$-simplex is the $d$-dimensional version of a triangle.

The problem of determining the minimum number of empty triangles any
set of $n$ points in general position in the plane contains, has been
widely studied~\cite{katmeir,smallnumber,fewemptydumi,minnumvaltr} and
also the higher dimensional version of the problem has been
considered~\cite{emptysimplices}. In~\cite{katmeir} it is noted that
every set of $n$ points in general position in $\mathbb{R}^d$
determines at least $\binom{n-1}{d}=\Omega(n^d)$ empty simplices. In
\cite{emptysimplices} it is shown that in a random set of $n$  points in 
$\mathbb{R}^d$---chosen uniformly at random on a convex, bounded set with nonempty
interior---the expected number of empty simplices is at most $c_d
\binom{n}{d}=O(n^d)$ (where $c_d$ is a constant depending only on $d$).

The colored version of the problem has been introduced
in~\cite{chromaticvariants} and was studied in
\cite{triangmonojournal}, where $\Omega(n^{5/4})$ empty monochromatic
triangles were shown to exist in every two colored set of $n$ points
in general position in the plane.  This has later been improved to
$\Omega(n^{4/3})$ in~\cite{pachmono}. Further, arbitrarily large
$3$-colored sets without empty monochromatic triangles were shown to
exist in the plane in~\cite{chromaticvariants}.

In this paper we study the higher dimensional version of this colored
variant. We generalize both, the dimension and the number of colors.
Specifically, we consider the problem of counting the number of empty
monochromatic $d$-simplices in a $k$-colored set of points in
$\mathbb{R}^d$.

It is shown in \cite{jorge3d} that every sufficiently large
$4$-colored set of points in general position in $\mathbb{R}^3$
contains an empty monochromatic tetrahedron. This is done by showing
that any set of $n$ points in general position in $\mathbb{R}^3$ can
be triangulated with more than $3n$ tetrahedra.

The problem of triangulating a set of points with many simplices is
intimately related to the problem of determining the minimum number of
empty simplices in $k$-colored sets of points in
$\mathbb{R}^d$. Remarkably this problem has received little attention.
For the special case of $\mathbb{R}^3$, it even has been pronounced
``the least significant'' among the four extremal (maxmax, maxmin,
minmax, minmin) problems in~\cite{epw-tpstd-90}. Consequently, only a
trivial lower bound and an upper bound of $\frac{7}{15}n^2+O(n)$ has
been shown there.
Nevertheless, in~\cite{pbrass} sets of $n$ points in $\mathbb{R}^d$ in
general position are shown such that every triangulation of them has
$O(n^{5/3})$ tetrahedra, for points in $\mathbb{R}^3$, and in general
$O(n^{1/d+ \lceil d/2 \rceil\cdot(d-1)/d})$ simplices for
points in $\mathbb{R}^d$.
Furthermore, in~\cite{bmj-rpdg-05} this minmax problem is stated as
Open Problem~11 in the section ``Extremal Number of Special
Subconfigurations''.

In this direction we give the first, although not asymptotically
improving, non-trivial lower bound and show that for $d \ge 3$ every
set of $n$ points in general position in $\mathbb{R}^d$ admits a
triangulation of at least $dn+\Omega(\log n)$ simplices, for $n$
sufficiently large and $d$ constant.

\begin{figure}[!p]
  \centering
  \includegraphics[scale=0.85]{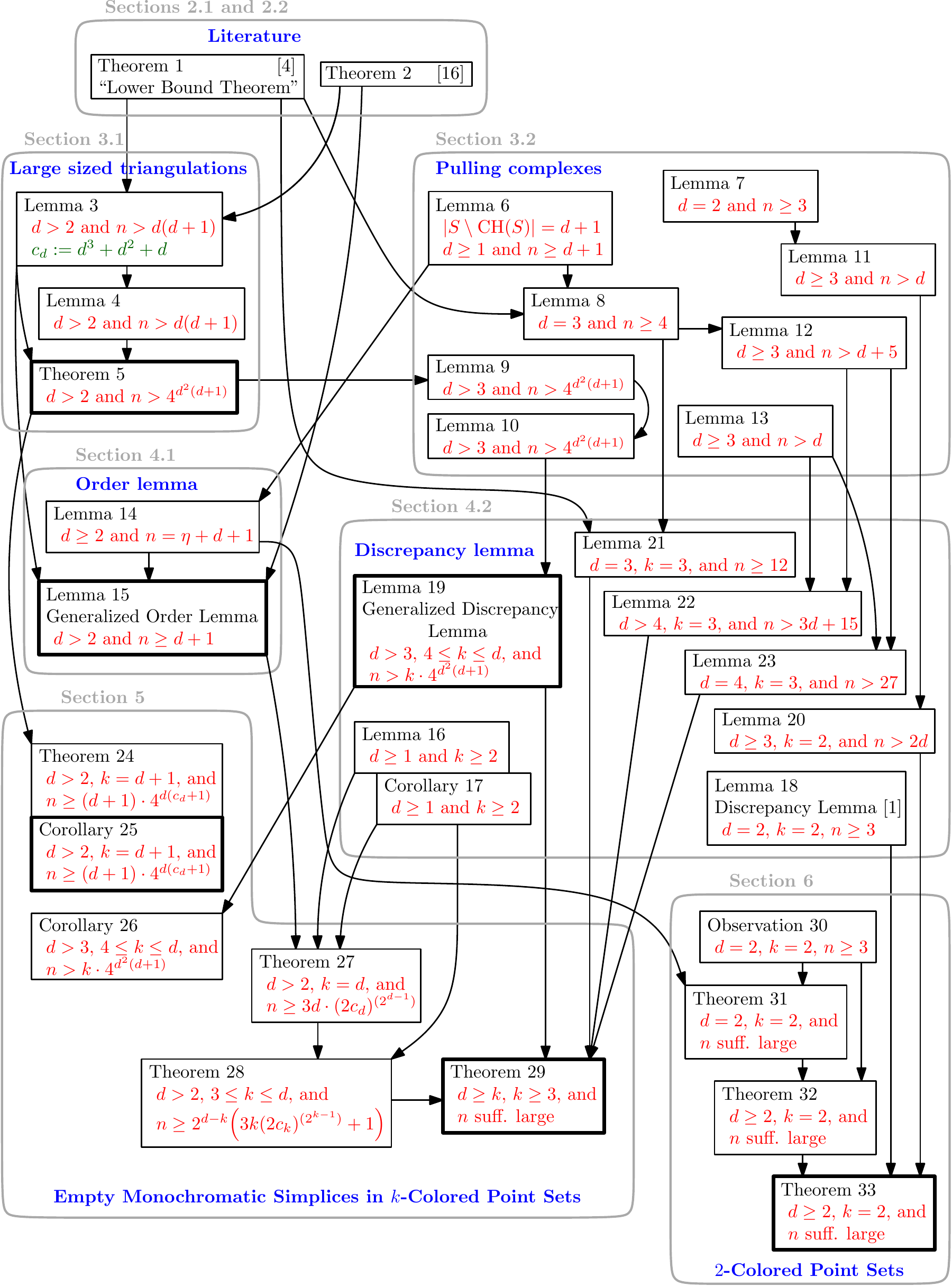}
  \caption{Roadmap through the paper.}
  \label{fig:roadmap}
\end{figure}

The paper is organized as follows: in Section~\ref{sec:prelim} known
results on simplicial complexes and triangulations are reviewed; in
Section~\ref{sec:complex_and_triang} new results on simplicial
complexes and triangulations are presented; using these results in
Section~\ref{sec:higher}, high dimensional versions of the Order and
Discrepancy Lemmas used in~\cite{triangmonojournal} are shown; in
Section~\ref{sec:ems} the lemmas of Section~\ref{sec:higher} are put
together to prove various results on the minimum number of empty
monochromatic simplices in sets of points in $\mathbb{R}^d$.
Our results are summarized in Table~\ref{tab:conclusion_k_over_d}.

\begin{table}[htb]
  \centering
  \begin{tabular}{@{}r|c|c@{}}
    & $d=2$ & $d\geq3$ \\ \hline
    \rule[0mm]{0mm}{3ex}$k=2$ & $\Omega(n^{4/3})$ (\cite{pachmono} and
    Thm~\ref{thm:ems_in_d2+k2})
    & $\Omega(n^{d-2/3})$ (Thm~\ref{thm:ems_in_d2+k2})  \\
    $3\leq k \leq d$ & --- & $\Omega(n^{d-k+1+2^{-d}})$ (Thm~\ref{thm:ems_in_d3+}) \\
    $k=d+1$ & none (\cite{chromaticvariants}) & at least linear$^*$ (Cor~\ref{cor:(d+1)})\\
    $k\geq d+2$ & none (\cite{chromaticvariants}) & unknown
  \end{tabular}
  \caption{Number of empty monochromatic $d$-simplices in $k$-colored sets
    of $n$ (sufficiently large) points in $\mathbb{R}^d$. $^*$ The
    linear lower bound for $d=3$ and $k=4$ has been proved already
    in~\cite{jorge3d}.}
  \label{tab:conclusion_k_over_d}
\end{table}

To provide a better general view on the paper, and especially to
visualize the interrelation between the many lemmas, we present a
``roadmap'' through the paper in \figurename~\ref{fig:roadmap}.
The lemmas (and theorems and corollaries) are shown in boxes, given
with their number, if applicable a special name, and the necessary
preconditions. Main results have a bold frame.
The lemmas are grouped to reflect their topical and section
correlation.
An arrow from a Lemma~A to a Lemma~B depicts, that the proof of
Lemma~B uses the result of Lemma~A. Hence, the preconditions for
Lemma~A have to be fulfilled in Lemma~B.
Theorem~\ref{thm:quadratic} (stated and proven in the ''Conclusions'') is not
depicted in \figurename~\ref{fig:roadmap}, as there is no
interrelation with other lemmas.

\section{Preliminaries}
\label{sec:prelim}

In this section, following the notation of Matou{\v s}ek
\cite{borsuk}, we state the definitions and known results regarding
simplicial complexes and triangulations, that will be needed
throughout the paper.
Note that in this paper we consider the number, $d$, of dimensions and
also the number, $k$, of different colors as constants. This means,
that $d$ and $k$ do not depend on the size, $n$, of the considered
finite set of points. But of course the required minimum size of the point set
might depend on $d$ and $k$.

\subsection{Simplicial Complexes}
\label{sec:simplcompl}

Let $X$ be a finite set of points in $\mathbb{R}^d$. The \emph{convex hull}
of $X$, denoted with $\Conv(X)$, is the intersection of all convex sets
containing $X$. Alternatively it may be defined as the set of points
that can be written as a \emph{convex combination} of elements of
$X$:
\[ \Conv(X) =\left\{\sum_{i=1}^{|X|} \alpha_i x_i \ \Bigg | \ x_i\in X, \,
  \alpha_i\in \mathbb{R}, \, \alpha_i \geq 0, \, \sum_{i=1}^{|X|}
  \alpha_i=1\,\right\} \mbox{ .}\]
We denote the boundary of $\Conv(X)$ with $\CH(X)$. A point of $X$ is
said to be a \emph{convex hull point} if it lies in $\CH(X)$,
otherwise it is called an \emph{interior point}.  A point set $X$ is
said to be in \emph{convex position} if every point of $X$ is a convex
hull point.

Let $\mathbf{0}$ denote the $d$-dimensional zero vector. A set of
points $\{x_1,\dots,x_n\}$ in $\mathbb{R}^d$ is said to be
\emph{affinely dependent} if there exist real numbers $(\alpha_1,
\dots, \alpha_n)$, not all zero, such that $\sum_{i=1}^n
\alpha_ix_i=\mathbf{0}$ and $\sum_{i=1}^n \alpha_i=0$. Otherwise
$\{x_1,\dots,x_n\}$ is said to be \emph{affinely independent}.
A set of points $X$ in $\mathbb{R}^d$ is in \emph{general position} if
each subset of $X$ with at most $d+1$ elements is affinely
independent.

A \emph{simplex} $\sigma$ is the convex hull of a finite affinely independent
set $A$ in $\mathbb{R}^d$. The elements of $A$ are called
the \emph{vertices} of $\sigma$. If $A$ consists of $m+1$ elements,
we say that $\sigma$ is of \emph{dimension} $\dim \sigma:=m$ or that
$\sigma$ is an $m$-simplex.
The convex hull of any subset of vertices of a simplex $\sigma$ is
called a \emph{face} of $\sigma$. A face of a simplex is again a
simplex.

A \emph{simplicial complex} $\mathcal{K}$ is a family of simplices
satisfying the following properties:
\begin{itemize}
\item Each face of every simplex in $\mathcal{K}$ is also a simplex of
  $\mathcal{K}$.
\item The intersection of two simplices $\sigma_1, \sigma_2 \in
  \mathcal{K}$ is either empty or a face of both, $\sigma_1$ and
  $\sigma_2$.
\end{itemize}

The \emph{vertex set} of $\mathcal{K}$ is the union of the vertex sets
of all simplices in $\mathcal{K}$.
We say that $\mathcal{K}$ is of \emph{dimension} $m$, if $m$ is the
highest dimension of any of its simplices. The \emph{size} of a
simplicial complex of dimension $m$ is the number of its simplices of
dimension $m$.
The \emph{$j$-skeleton} of $\mathcal{K}$ is the simplicial complex
consisting of all simplices of $\mathcal{K}$ of dimension at
most~$j$. Hence the $0$-skeleton is the vertex set of $\mathcal{K}$.

We now turn to finite sets of points in general position in
$\mathbb{R}^d$. Let $S$ be such a set of $n$ elements.
Note that since $S$ is in general position we may regard $\CH(S)$ as a
simplicial complex in a natural way. Such simplicial complexes are
called \emph{simplicial polytopes}.
It is known that every simplicial polytope satisfies:

\begin{theorem}[\cite{lowerbound}~Lower Bound Theorem]
  \label{thm:lowerbound}
  For a simplicial polytope of dimension $d$ let $f_m$ be the number
  of its $m$-dimensional faces. Then:
\begin{itemize}
\item $f_m \ge \binom{d}{m}f_0-\binom{d+1}{m+1}m$ \ for all \ $1 \le m
  \le d-2$ \ and
\item $f_{d-1} \ge (d-1)f_0-(d+1)(d-2)$ \ .
\end{itemize} 
\end{theorem}

Note that in the Lower Bound Theorem, the word \emph{dimension} refers
to the dimension of the simplicial polytope as a polytope. Hence, a
three dimensional simplicial polytope would be a two dimensional
simplicial complex.

\subsection{Triangulations}

A \emph{triangulation} $\mathcal{T}$ of $S$ is a simplicial complex
such that its vertex set is $S$ and the union of all simplices of
$\mathcal{T}$ is $\Conv(S)$. This definition generalizes the usual
definition of triangulations of planar point sets. The \emph{size} of
a triangulation is the number of its $d$-simplices.
The minimum size of any triangulation of $S$ is known to be $n-d$. We
explicitly mention this result for further use:

\begin{theorem}[\cite{n-d}]\label{thm:n-d}
  Every triangulation of a set of $n$ points in general position in
  $\mathbb{R}^d$ has size at least $n-d$.
\end{theorem}

\begin{figure}[htb]
  \centering
  \includegraphics{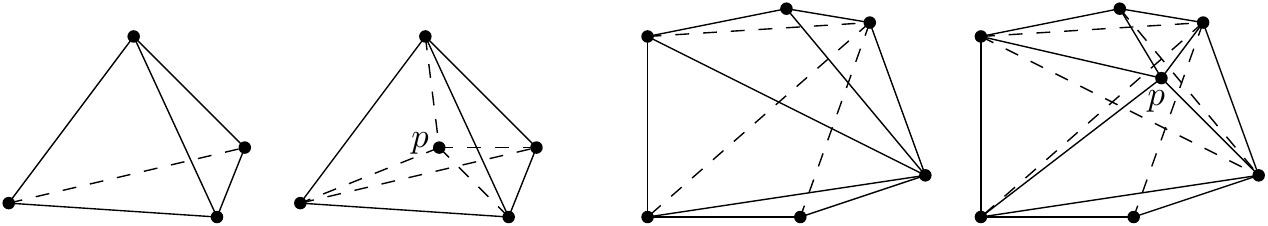}
  \caption{Example in $\mathbb{R}^3$ for inserting a point $p$ into a
    triangulation with (left) $p$ inside the convex hull and (right)
    $p$ outside the convex hull.}
  \label{fig:ins}
\end{figure}

We will use the following operation of \emph{inserting a point $p$
  into a triangulation $\mathcal{T}$} frequently:
Let $p$ be a point not in $S$ but such that $S \cup \{p\}$ is also in
general position, and let $\mathcal{T}$ be a triangulation of $S$.
If $p$ lies in $\Conv(S)$ then $p$ is contained in a unique
$d$-simplex $\sigma$ of $\mathcal{T}$. We remove $\sigma$ from
$\mathcal{T}$ and replace it with the $(d+1)$ $d$-simplices formed by
taking the convex hull of $p$ and each of the \mbox{$(d+1)$} \
$(d-1)$-dimensional faces of $\sigma$.
If, on the other hand, $p$ lies outside $\Conv(S)$ then a set
$\mathcal{F}$ of $(d-1)$-dimensional faces of $\CH(S)$ is visible from
$p$. We get a set of $d$-simplices formed by
taking the convex hull of $p$ and each face of $\mathcal{F}$, and add
these simplices to $\mathcal{T}$.
In either case the resulting family of simplices is a triangulation of
$S \cup \{p\}$ (see Figure \ref{fig:ins}).

We distinguish two different types of triangulations of a set $S$ of $n$
points in general position in $\mathbb{R}^d$ by their construction:
A \emph{shelling triangulation} of $S$ is constructed as
follows. Choose any ordering $p_1,p_2, \dots, p_n$ of the elements of
$S$ and let $S_i=\{p_1,\dots,p_i\}$. Start by triangulating $S_{d+1}$
with only one simplex. Afterwards, for every $i>d+1$ create the
triangulation of $S_{i}$ by inserting $p_{i}$ into the
triangulation of $S_{i-1}$. The final triangulation of this process, that
of $S_n$, is a shelling triangulation.
A \emph{pulling triangulation} of $S$ is constructed by choosing (if
it exists) a point $p$ of $S$, such that
\mbox{$S\!\setminus\!((\CH(S)\cap S)\cup\{p\})=\emptyset$}. Then
\mbox{$S\! \setminus\! \{p\}$} is in convex position. Construct a
$d$-simplex with $p$ and each $(d-1)$-dimensional face of $\CH(S)$
that does not contain~$p$.

\section{Results on Triangulations and Simplicial Complexes}
\label{sec:complex_and_triang}

In this section we present some results on triangulations and
simplicial complexes that will be needed later, but are also of
independent interest.
We begin by showing that every point set can be triangulated with a
``large number'' of simplices. We use the same strategy as
in~\cite{jorge3d}.

\subsection{Large Sized Triangulations}

First we prove an at least possible size for a triangulation of a
convex set of points, by building a shelling triangulation for a
special sequence of points.

\begin{lemma}\label{lem:convtriang}
  Every set $S$ of $n>d(d+1)$ points in convex and general position in
  $\mathbb{R}^d$ $(d>2)$ has a triangulation of size at least
  $(d+1)n-c_d$, with $c_d=d^3+d^2+d$.
\end{lemma}

\begin{proof}
  The $1$-skeleton of $\CH(S)$ is a graph of $n$ vertices
  and, by the Lower Bound Theorem (Theorem~\ref{thm:lowerbound}, for
  $m=1$), of at least $dn-\frac{d(d+1)}{2}$ edges. Therefore, as long
  as $n>d(d+1)$ there will be a vertex of degree at least $2d$ in this
  graph. 

  Set $S_n:=S$ and let $G_n$  be the $1$-skeleton (as a graph)
  of $\CH(S_n)$. In general once $S_i$ is defined,
  let $G_i$ be the $1$-skeleton (as a graph) of $\CH(S_i)$. Let
  $p_i$ be a vertex of degree at least $2d$ in $G_i$, with $n\geq i >
  d(d+1)$.
  We construct a shelling triangulation $\mathcal{T}_n$ of $S_n$, with
  size as claimed in the lemma.

  Starting with $S_n$, iteratively remove a vertex $p_i$ from $S_i$,
  i.e., $S_{i-1} = S_i \setminus \{p_i\}$. Observe that $|S_i| =
  i$. The iteration stops with $S_{i-1} = S_{d(d+1)}$ as $i > d(d+1)$.
  Construct an arbitrary shelling triangulation $\mathcal{T}_{d(d+1)}$
  of $S_{d(d+1)}$. By Theorem~\ref{thm:n-d}, $\mathcal{T}_{d(d+1)}$ has
  size at least $d(d+1)-d=d^2$. Complete $\mathcal{T}_{d(d+1)}$ to a
  shelling triangulation $\mathcal{T}_n$ by inserting the points $p_i$
  in reversed order of their removal ($i$ from $d(d+1)+1$ to $n$). 

  We prove that with each inserted point $p_i$ at least $(d+1)$
  $d$-simplices are added to the triangulation.
  Let $\varrho_i$ be the degree of $p_i$ in $G_i$ and recall that $\varrho_i
  \geq 2d$. Consider the neighbors $q_1,\dots,q_{\varrho_i}$ of $p_i$ in
  $G_i$.
  Let $\Pi$ be a $(d-1)$-dimensional hyperplane separating $p_i$ and
  $S_{i-1}$, and let $q_1',\dots,q_{\varrho_i}'$ be the set of
  intersections of $\Pi$ with the lines spanned by $p_i$ and each of
  $q_1,\dots,q_{\varrho_i}$.

  Note that $q_1',\dots,q_{\varrho_i}'$ are a set of points in convex
  position in $\mathbb{R}^{d-1}$ and that the $(d-1)$-dimensional
  faces of $\CH(S_{i-1})$, which are visible to $p_i$, project to a
  triangulation of $q_1',\dots,q_{\varrho_i}'$ in~$\Pi$.
  By Theorem \ref{thm:n-d}, every triangulation of $\varrho_i$ points in
  $\mathbb{R}^{d-1}$ has size at least $\varrho_i-(d-1) \ge d+1$.
  Thus, at least $(d+1)$ $d$-simplices are added when inserting
  $p_i$.
  Hence, the constructed shelling triangulation $\mathcal{T}_n$ has
  size at least $d^2+(d+1)(n-d(d+1))$, which is the claimed bound of
  $(d+1)n-c_d$, with $c_d=d(d+1)^2-d^2=d^3+d^2+d$.
\end{proof}

Using this result it is easy to give a lower bound on the
triangulation size for general point sets in dependence of a certain
subset property.

\begin{lemma}\label{lem:nested}
  Let $S$ be a set of points in general position in $\mathbb{R}^d$
  $(d>2)$. Let $P$ and $Q$ be two disjoint sets, such that $S=P\cup Q$
  and $Q$ is in convex position.
  If $|Q|>d(d+1)$ then there exists a triangulation of $S$ of size at
  least $(d+1)|Q|+|P|-c_d$, with $c_d$ defined as in
  Lemma~\ref{lem:convtriang}.
\end{lemma}

\begin{proof}
  By Lemma~\ref{lem:convtriang}, $Q$ has a triangulation $\mathcal{T}$
  of size at least $(d+1)|Q|-c_d$, if $|Q|>d(d+1)$. Inserting each
  point of $P$ into $\mathcal{T}$ adds at least one $d$-simplex to
  $\mathcal{T}$ per point in $P$. This results in a triangulation of
  $S$ with size at least $(d+1)|Q|+|P|-c_d$.
\end{proof}

Combining the previous two lemmas we prove a new non-trivial lower
bound for the size of triangulations with an additive logarithmic
term.

\begin{theorem}\label{thm:dn+log}
  Every set $S$ of $n>4^{d^2(d+1)}$ points in general position in
  $\mathbb{R}^d$ $(d>2)$, with $h$ convex hull points, has a
  triangulation of size at least
  $dn+\max{\left\{h,\frac{\log_2(n)}{2d}\right\}}-c_d$, with $c_d$ as
  defined in Lemma~\ref{lem:convtriang}.
\end{theorem}

\begin{proof}
  Let $P$ be the set of convex hull points of $S$. We distinguish two
  cases:
  \begin{itemize}
  \item{$|P|=h > \log_2(n)/(2d)$.} By Lemma~\ref{lem:convtriang}, there
    exists a triangulation of $P$ of size at least $(d+1)h-c_d$, as
    $h > d(d+1)$.
    Insert the remaining $n-h$ points of $S \setminus P$ into this
    triangulation. Since these points are inside $\Conv(P)$, each of
    them contributes with $d$ additional $d$-simplices to the final
    triangulation. Therefore, the resulting triangulation has size at
    least $dn+h-c_d > dn+\frac{\log_2(n)}{2d}-c_d$.

  \item{$|P|=h \leq \log_2(n)/(2d)$.} By the Erd\H os-Szekeres
    Theorem (see~\cite{happyend}) and its best known upper bound
    (see~\cite{upperes}), $S$ contains a subset $Q$ of at least $|Q| >
    \frac{\log_2(n)}{2} > d(d+1)$ points in convex position.
    Let $P'=P \setminus Q$. Apply Lemma~\ref{lem:nested} to obtain a
    triangulation $\mathcal{T}$ of $P'\cup Q$ of size at least
    $(d+1)|Q|+|P'|-c_d$. Insert the remaining points of $S \setminus
    (P' \cup Q)$ into $\mathcal{T}$.
    Since these inserted points are in the interior of $\Conv(P' \cup
    Q)$, each of them contributes with $d$ additional $d$-simplices to
    the final triangulation. Therefore, this triangulation has size at
    least $d(n-|Q|-|P'|)+(d+1)|Q|+|P'|-c_d = dn+|Q|-(d-1)|P'|-c_d > 
     dn+\frac{\log_2(n)}{2}-(d-1)\frac{\log_2(n)}{2d}-c_d\geq 
    dn+\frac{\log_2(n)}{2}-\frac{\log_2(n)}{2}+\frac{\log_2(n)}{2d}-c_d$,
    which is $dn+\frac{\log_2(n)}{2d}-c_d$.
  \end{itemize}
\end{proof}

Note that $c_d$ in Lemma~\ref{lem:convtriang} can be improved to
$\frac{d(d+1)^2}{2}+\frac{d(d+1)}{12} =
\frac{d^3}{2}+\frac{13d^2}{12}+\frac{7d}{12}$.
Instead of stopping the process at $S_{d(d+1)}$, we
continue the iteration using a vertex degree of \mbox{$\,2d\!-\!1\,$} for $S_i$ with
$d(d+1)\geq i > \frac{d(d+1)}{2}$, a vertex degree of \mbox{$\,2d\!-\!2\,$} for $S_i$
with $\frac{d(d+1)}{2}\geq i > \frac{d(d+1)}{3}$, and so on.
This way, instead of a triangulation of size at least $d^2$, we can
guarantee a triangulation $\mathcal{T}_{d(d+1)}$ of size at least
$\sum_{i=1}^d\left(\left(2d-i-(d-1)\right)\frac{d(d+1)}{i(i+1)}\right)
\geq \frac{3}{4}d(d+1)^2-d(d+1)\cdot
\min{\left\{\frac{d+1}{4}+\frac{1}{12},\ln{(d+1)}\right\}}$, which
results in the claimed improvement of $c_d$ for $d\geq3$.
Thus, for $d=3$ Theorem~\ref{thm:dn+log} can be improved to
$3n+\max{\left\{h,\frac{\log_2{n}}{6}\right\}}-25$.
Note that this corresponds to the bound from~\cite{epw-tpstd-90}, that
every set of $n$ points in general position in $\mathbb{R}^3$, with
$h$ convex hull points, has a tetrahedrization of size at least
$3(n-h)+4h-25$ for $h\geq13$.

\subsection{Pulling Complexes}

Let $S$ be a set of $n$ points in general position in $\mathbb{R}^d$.
In this section we present lemmas that allow us to construct
$d$-simplicial complexes of large size on $S$, such that their
$d$-simplices contain a pre-specified subset of $S$ in their vertex
set.
We begin with a result for point sets, whose convex hull is a simplex.

\begin{lemma}\label{lem:starsimplex3d}
  Let $S$ be a set of $n \geq d+1$ points in general position in
  $\mathbb{R}^d$ $(d\geq 1)$, such that $\Conv(S)$ is a $d$-simplex.
  For every convex hull point $p$ of $S$, there exists a triangulation
  of $S$ such that $(d-1)n-d^2+2$ of its $d$-simplices have $p$ as a
  vertex.
\end{lemma}

\begin{figure}[htb]
  \centering
  \includegraphics{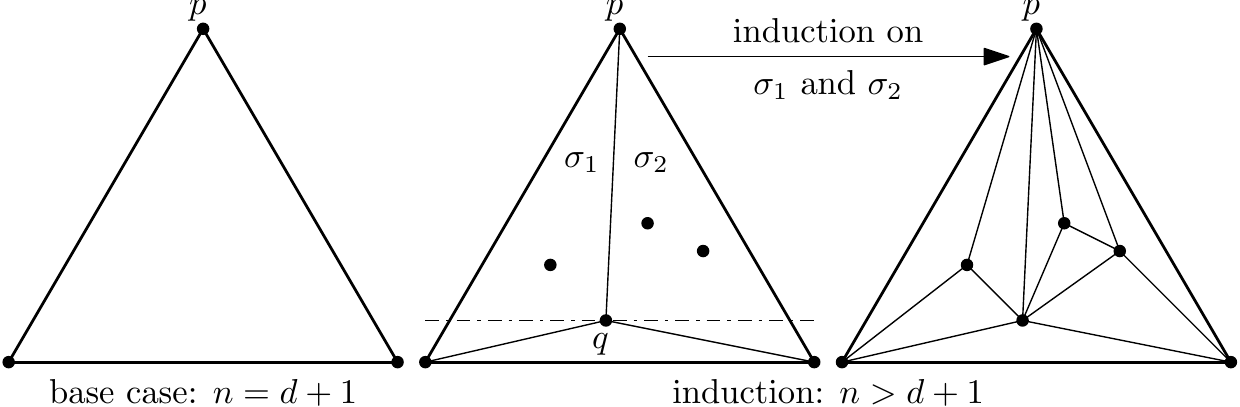}
  \caption{Illustration of the proof of Lemma~\ref{lem:starsimplex3d}
    for $n=7$ and $d=2$.}
  \label{fig:CH_simplex_2D}
\end{figure}

\begin{proof}
  We use induction on $n$, see Figure~\ref{fig:CH_simplex_2D} for an
  illustration. Start with a triangulation $\mathcal{T}$ consisting
  only of the $d$-simplex $\Conv(S)$. If $n=(d+1)$, $\mathcal{T}$ is a
  triangulation with $(d-1)n-d^2+2 = (d-1)(d+1)-d^2+2 = 1$ empty
  simplex containing $p$ as vertex.

  Assume $n>d+1$. Let $q$ be the interior point of $S$ closest to the
  only face of $\Conv(S)$ not incident to $p$.
  (If there exist more then one such closest points, then choose an
  arbitrary one of them as $q$.)
  Insert $q$ into $\mathcal{T}$. This results in a triangulation of
  size $(d+1)$ in which $d$ of its $d$-simplices,
  $\sigma_1\ldots\sigma_d$, have $p$ as a vertex.  Note that the
  remaining $d$-simplex does not contain any point of $S$ in its
  interior.
  We apply induction on $\sigma_1\ldots\sigma_d$. Let $n_i$ $(1\leq i
  \leq d)$ be the number of points of $S$ interior to $\sigma_i$,
  $\sum_{i=1}^d n_i = n - (d+1) - 1$.
  For each $\sigma_i$ we obtain a triangulation such that
  $(d-1)(n_i+(d+1))-d^2+2$ of its $d$-simplices have $p$ as a
  vertex. The union of the triangulations of each $\sigma_i$ is a
  triangulation of $S$, and $\sum_{i=1}^d ((d-1)n_i+(d-1)(d+1)-d^2+2)
  = (d-1)\sum_{i=1}^d (n_i) + d = (d-1)n-d^2+2$ of its $d$-simplices
  have $p$ as a vertex.
\end{proof}

The next three lemmas give, for every point of a general point set in
$\mathbb{R}^d$, a lower bound on the number of interior disjoint
$d$-simplices incident to $p$, for the cases $d=2$, $d=3$, and $d>3$,
respectively.

\begin{lemma}\label{lem:starsimplex2d}
  Let $S$ be a set of $n \geq 3$ points in general position in
  $\mathbb{R}^2$. For every point $p$ of $S$ there exists a
  $2$-dimensional simplicial complex of size at least $(n-2)$ and such
  that all of its triangles have $p$ as a vertex.
\end{lemma}

\begin{proof}
  Do a cyclic ordering around $p$ of the points of
  \mbox{$S\!\setminus\!\!\{p\}$}. Construct a $2$-dimensional simplicial
  complex by forming a triangle with $p$ and every two consecutive
  elements determining an angle less than $\pi$. This simplicial
  complex has at least $n-2$ triangles and they all contain $p$ as a
  vertex.
\end{proof}

\begin{lemma}\label{lem:starp3d}
  Let $S$ be a set of $n \geq 4$ points in general position in
  $\mathbb{R}^3$. For every point $p$ of $S$ there exists a
  triangulation of $S$ such that at least:
  \begin{itemize}
  \item $2n-6$ of its $3$-simplices have $p$ as a vertex, if $p$ is
    an interior point of $S$.
  \item $2n-\varrho(p)-4$ of its $3$-simplices contain $p$ as a
    vertex, if $p$ is a convex hull point of $S$ and $\varrho(p)$ is
    its degree in the $1$-skeleton of $\CH(S)$.
  \end{itemize}
\end{lemma}

\begin{proof}
  Let $S'$ be the set of convex hull points of $S$ and
  $n'=|S'|$. Construct a pulling triangulation $\mathcal{T}'$
  w.r.t. $p$ of $S'\cup \{p\}$. By definition all $3$-simplices of
  $\mathcal{T}'$ contain $p$ as a vertex. For every $3$-simplex
  $\sigma$ of $\mathcal{T}'$, let $\eta$ be the number of points
  of $S$ interior to $\sigma$.
  By applying Lemma~\ref{lem:starsimplex3d} we can triangulate
  $\sigma$, such that $2(\eta+4)-7 = 2\eta+1$ of its
  $3$-simplices have $p$ as a vertex. Repeat this for every
  $3$-simplex of $\mathcal{T}'$, to obtain a triangulation
  $\mathcal{T}$ of $S$.

  By Theorem~\ref{thm:lowerbound}, $\CH(S)$ has (at least) $2n'-4$
  faces (for $d=3$ this lower bound is tight).
  \begin{itemize}
  \item If $p$ is an interior point of $S$, $\mathcal{T}'$ contains a
    $3$-simplex for every face of $\CH(S)$. Therefore, summing over
    all these faces we get $\sum (2\eta+1) = 2\sum
    (\eta)+2n'-4 = 2(n-n'-1)+2n'-4 = 2n-6$ of the $3$-simplices in
    $\mathcal{T}$ have $p$ as a vertex.
  \item If $p$ is a convex hull point of $S$, $\mathcal{T}'$ contains
    a $3$-simplex for every face of $\CH(S)$ not having $p$ as a
    vertex. This is equal to $2n'-4-\varrho(p)$, where $\varrho(p)$ is the degree
    of $p$ in the $1$-skeleton of $\CH(S)$.
    Therefore, $\sum (2\eta+1) = 2\sum (\eta)+2n'-4-\varrho(p)
    = 2(n-n')+2n'-4-\varrho(p) = 2n-\varrho(p)-4$ of the $3$-simplices
    in $\mathcal{T}$ have $p$ as a vertex.
  \end{itemize}\vspace*{-7mm}
\end{proof}

\begin{lemma}\label{lem:starp}
  Let $S$ be a set of $n>4^{d^2(d+1)}$ points in general position in
  $\mathbb{R}^d$ $(d>3)$. For every point $p$ of $S$, there exists a
  $d$-dimensional simplicial complex $\mathcal{K}$ with vertex set
  $S$, such that $\mathcal{K}$ has size strictly larger than
  $(d-1)n+\frac{\log_2 n}{2(d-1)}-2c_{d-1}$ and all its $d$-simplices
  have $p$ as a vertex, with $c_d$ defined as in
  Lemma~\ref{lem:convtriang}.
\end{lemma}

\begin{proof}
  For every point $q \in S$ distinct from $p$ let $r_{q}$ be the
  infinite ray with origin $p$ and passing through $q$. Let $\Pi$ be a
  halving $(d-1)$-dimensional hyperplane of $S$ passing through $p$,
  not containing any other point of $S$.
  Further, let $\Pi_1$ and $\Pi_2$ be two $(d-1)$-dimensional
  hyperplanes parallel to $\Pi$ containing $\Conv(S)$ between them and
  not parallel to any of the rays $r_q$.

  Project from $p$ every point in \mbox{$S\!\setminus\!\!\{p\}$} to
  $\Pi_1$ or $\Pi_2$, in the following way. Every ray $r_q$ intersects
  either $\Pi_1$ or $\Pi_2$ in a point $q'$. Take $q'$ to be the
  projection of $q$ from $p$. Let $S_1'$ and $S_2'$ be these projected
  points in $\Pi_1$ and $\Pi_2$, respectively.
  Both, $S_1'$ and $S_2'$, are sets of points in general position in
  $\mathbb{R}^{d-1}$, with
  $|S_1'|=n_1=\left\lfloor\frac{n-1}{2}\right\rfloor$ and
  $|S_2'|=n_2=\left\lceil\frac{n-1}{2}\right\rceil$, where both,
  $n_1$ and $n_2$, are strictly larger than $4^{(d-1)^2d}$.

  By Theorem~\ref{thm:dn+log}, there exist triangulations
  $\mathcal{T}_1$ of $S_1'$ and $\mathcal{T}_2$ of $S_2'$ of size at least
  $(d-1)n_1+\frac{\log_2(n_1)}{2(d-1)}-c_{d-1}$  and
  $(d-1)n_2+\frac{\log_2(n_2)}{2(d-1)}-c_{d-1}$, respectively.
  Consider the simplicial complexes $\mathcal{K}_1$ and
  $\mathcal{K}_2$ that arise from replacing every point $q'$ in a
  simplex of $\mathcal{T}_1$ or $\mathcal{T}_2$ with its preimage $q$
  in $S\setminus \{p\}$. The $(d-1)$-simplices of $\mathcal{K}_1$ and
  $\mathcal{K}_2$ are all visible from $p$. Hence, we obtain a
  simplicial complex $\mathcal{K}$ of dimension~$d$, by taking the
  convex hull of $p$ and each $(d-1)$-simplex of $\mathcal{K}_1$ and
  $\mathcal{K}_2$.
  Obviously, all $d$-simplices of $\mathcal{K}$ contain $p$ as a
  vertex. The size of $\mathcal{K}$ is at least
  $(d-1)(n_1+n_2) + \frac{\log_2(n_1)+\log_2(n_2)}{2(d-1)}-2c_{d-1}
  = (d-1)n - (d-1) + \frac{\log_2(n_1n_2)}{2(d-1)} - 2c_{d-1} \geq \linebreak
  (d-1)n + \frac{\log_2(\frac{n(n-2)}{4})}{2(d-1)} - 2c_{d-1} -
  (d-1)
  = (d-1)n + \frac{\log_2(n)}{2(d-1)} - 2c_{d-1} +
  \frac{\log_2(n-2)-\log_2(4)}{2(d-1)} - (d-1)
  > (d-1)n + \frac{\log_2(n)}{2(d-1)} - 2c_{d-1} +
  \frac{2d^2(d+1)-1-2}{2(d-1)} - (d-1)
  > (d-1)n + \frac{\log_2(n)}{2(d-1)} - 2c_{d-1} + (d-1)^2$.
  This is strictly larger than
  $(d-1)n+\frac{\log_2(n)}{2(d-1)}-2c_{d-1}$.
\end{proof}

We now consider not only one point, but subsets $X$ of point sets in
$\mathbb{R}^d$ $(d>3)$. The next three lemmas, applicable for
$1\leq|X|\leq d-3$, $|X|=d-1$, and $|X|=d-2$, respectively, provide
lower bounds on the number of interior disjoint $d$-simplices which
all share the points in $X$. Note that the second lemma in the row,
Lemma~\ref{lem:starset2r}, is true for $d\geq3$.

\begin{lemma}\label{lem:starset}
  Let $S$ be a set of $n>4^{d^2(d+1)}$ points in general position in
  $\mathbb{R}^d$ $(d>3)$.
  For every set $X \subset S$ of $r$ points $(1\leq r\leq d-3)$, there
  exists a $d$-dimensional simplicial complex $\mathcal{K}$ with
  vertex set $S$, such that $\mathcal{K}$ has size strictly larger
  than $(d-r)n+\frac{\log_2 n}{2(d-r)}-2c_{d-1}$ and all its
  $d$-simplices have $X$ in their vertex set, with $c_d$ defined as in
  Lemma~\ref{lem:convtriang}.
\end{lemma}

\begin{proof}
  The case $r=1$ is shown in Lemma~\ref{lem:starp}.
  Thus assume that $r>1$.
  Let $\Pi$ be the $(r-1)$-dimensional hyperplane containing $X$ and
  let $\Pi'$ be a $(d-(r-1))$-dimensional hyperplane orthogonal to
  $\Pi$.
  Project $S$ orthogonally to $\Pi'$, and let $S'$ be the resulting
  image. The set $X$ is projected to a single point $p_X$ in
  $\Pi'$. Obviously $|S'| = n-r+1 > 4^{(d-r+1)^2(d-r+2)}$. Apply
  Lemma~\ref{lem:starp} to $S'$, and obtain a $(d-r+1)$-dimensional
  simplicial complex $\mathcal{K}'$ with vertex set $S'$ of size at
  least $(d-r)(n-r+1)+\frac{\log_2 (n-r+1)}{2(d-r)}-2c_{d-r} =
  (d-r)n+\frac{\log_2 n}{2(d-r)}-2c_{d-r}-(d-r)(r-1)+\frac{\log_2
    (1-\frac{r-1}{n})}{2(d-r)} >
  (d-r)n+\frac{\log_2 n}{2(d-r)}-2c_{d-1}$, such that all the
  $(d-r+1)$-simplices of $\mathcal{K}'$ have $p_X$ as a vertex.

  To get $\mathcal{K}$ from $\mathcal{K}'$, lift each simplex of
  $\mathcal{K}'$ to the convex hull of the preimage of its vertex
  set. Thus $\mathcal{K}$ is a $d$-dimensional simplicial complex with
  vertex set $S$ and size larger than $(d-r)n+\frac{\log_2
    n}{2(d-r)}-2c_{d-1}$. As all $(d-r+1)$-simplices of $\mathcal{K}'$
  have $p_X$ as a vertex, each $d$-simplex of $\mathcal{K}$ has $X$ as
  a vertex subset.
\end{proof}

\begin{lemma}\label{lem:starset2r}
  Let $S$ be a set of $n > d$ points in general position in
  $\mathbb{R}^d$ $(d\geq3)$. For every set $X \subset S$ of $d-1$ points,
  there exists a $d$-dimensional simplicial complex $\mathcal{K}$ with
  vertex set $S$, such that $\mathcal{K}$ has size at least $n-d$, and
  all $d$-simplices of $\mathcal{K}$ have $X$ as a vertex.
\end{lemma}

\begin{figure}[htb]
  \centering
  \includegraphics{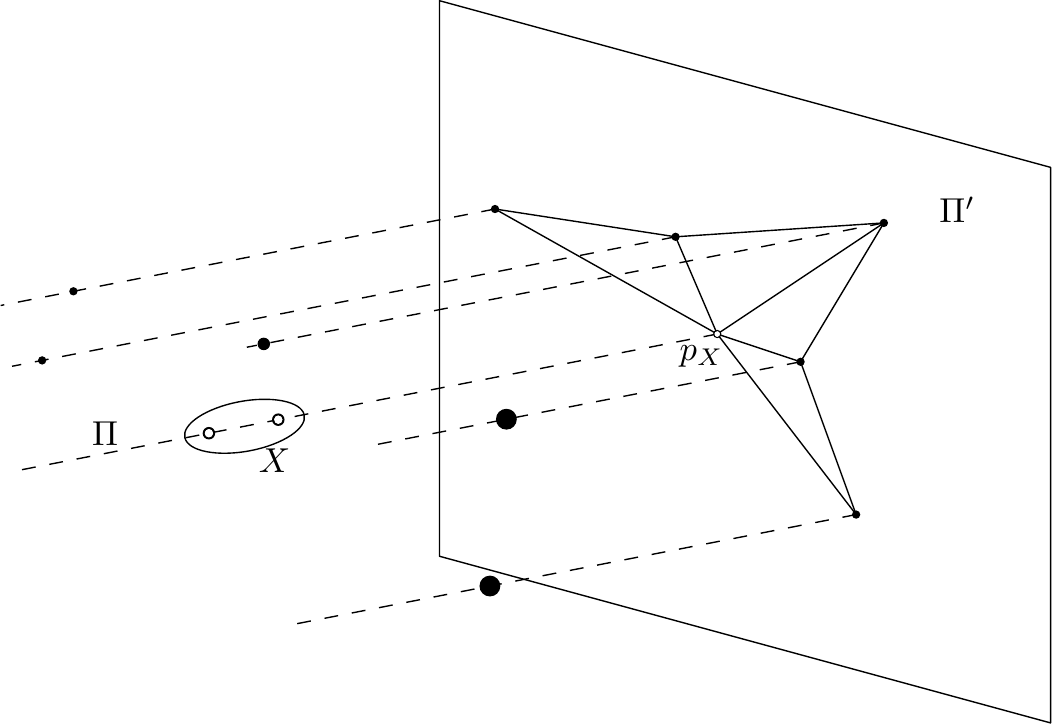}
  \caption{Illustration of the proof of Lemma~\ref{lem:starset2r} for
    $n=7$ and $d=3$.}
  \label{fig:demproy}
\end{figure}

\begin{proof}
  The proof is similar to that of Lemma~\ref{lem:starset}, with the
  difference that we cannot apply Lemma~\ref{lem:starp}.

  Let $\Pi$ be the $(d-2)$-dimensional hyperplane containing $X$ and
  let $\Pi'$ be a $2$-dimensional hyperplane orthogonal to $\Pi$.
  Project $S$ orthogonally to $\Pi'$, and let $S'$ be its image. The
  set $X$ is projected to a single point $p_X$ of $\Pi'$ (see
  Figure~\ref{fig:demproy}). Obviously $|S'| = n-d+2 \geq 3$.
  Apply Lemma~\ref{lem:starsimplex2d} to $S'$, and obtain a
  $2$-dimensional simplicial complex $\mathcal{K}'$ with vertex set
  $S'$ of size at least $(n-d+2)-2 = n-d$, such that all triangles of
  $\mathcal{K}'$ have $p_X$ as a vertex.

  To get $\mathcal{K}$ from $\mathcal{K}'$, lift each triangle of
  $\mathcal{K}'$ to the convex hull of the preimage of its vertex
  set. Thus $\mathcal{K}$ is a $d$-dimensional simplicial complex with
  vertex set $S$ and size $n-d$. Since all triangles of $\mathcal{K}'$
  have $p_X$ as a vertex, all $d$-simplices of $\mathcal{K}$ have $X$
  as a vertex subset.
\end{proof}

Note that Lemma~\ref{lem:starset} and Lemma~\ref{lem:starset2r} leave
a gap for $r=d-2$. In this case, the point set is projected to a
3-dimensional hyperplane, where the guaranteed bounds on incident
3-simplices vary significantly for extremal and interior points, see
Lemma~\ref{lem:starp3d}. Thus we make a weaker statement for this
case, which will turn out to be sufficient anyhow.

\begin{lemma}\label{lem:starset_d-2}
  Let $S$ be a set of $n>d+5$ points in general position in
  $\mathbb{R}^d$ $(d>3)$.
  Let \mbox{$X \subset S$} be a subset of \mbox{$d\!-\!2$}
  points. Denote with $\Pi$ the \mbox{$(d\!-\!3)$-dimensional}
  hyperplane containing $X$ and with $\Pi'$ a $3$-dimensional
  hyperplane orthogonal to $\Pi$.
  Project $S$ orthogonally to $\Pi'$, and let $S'$ be the resulting
  image. The set $X$ is projected to a single point $p_X$ in
  $\Pi'$.

  If $p_X$ is an interior point of $S'$, then there exists a
  $d$-dimensional simplicial complex $\mathcal{K}$ with vertex set
  $S$, such that $\mathcal{K}$ is of size at least $2n-2d-8$ and all
  $d$-simplices of $\mathcal{K}$ have $X$ in their vertex set.
\end{lemma}

\begin{proof}
  Obviously $|S'| = n-d-1 > 4$. As $p_X$ is assumed to be an interior
  point of $S'$, apply Lemma~\ref{lem:starp3d} to $S'$, and obtain a
  $3$-dimensional simplicial complex $\mathcal{K}'$ with vertex set
  $S'$ of size at least $2(n-d-1)-6 = 2n-2d-8$, such that all the
  $3$-simplices of $\mathcal{K}'$ have $p_X$ as a vertex.

  To get $\mathcal{K}$ from $\mathcal{K}'$, lift each $3$-simplex of
  $\mathcal{K}'$ to the convex hull of the preimage of its vertex
  set. Thus $\mathcal{K}$ is a $d$-dimensional simplicial complex with
  vertex set $S$ and size at least $2n-2d-8$. As all $3$-simplices of
  $\mathcal{K}'$ have $p_X$ as a vertex, each $d$-simplex of
  $\mathcal{K}$ has $X$ as a vertex subset.
\end{proof}

In the light of the previous lemma it is of interest to know the
conditions for a subset $X$ of $S$ in $\mathbb{R}^d$ $(d>3)$ to
project to an interior point of $S'$. We make the following statement.

\begin{lemma}\label{lem:extremeprojection}
  Let $S$ be a set of $n>d$ points in $\mathbb{R}^d$ $(d>3)$ and let
  $X\subset S$ be a subset of \mbox{$d\!-\!2$}~points.
  With $\Pi$ denote the \mbox{$(d\!-\!3)$-dimensional} hyperplane
  spanned by $X$ and with $\Pi'$ a $3$-dimensional hyperplane
  orthogonal to $\Pi$.
  Project $S$ orthogonally to $\Pi'$ and denote with $S'$ the
  resulting image of $S$ and with $p_X$ the image of $X$,
  respectively.
  Then $p_X$ is an extremal point of $S'$ if and only if $\Conv(X)$ is
  a \mbox{$(d\!-\!3)$-dimensional} facet of $\CH(S)$.
\end{lemma}

\begin{proof}
  If $\Conv(X)$ is a \mbox{$(d\!-\!3)$-dimensional} facet of $\CH(S)$,
  then there exists a \mbox{$(d\!-\!1)$-dimensional} hyperplane
  $\Pi_T$ ``tangential'' to $\Conv(S)$, containing only $X$ and having
  all other points of $S$ on one side. Thus, there exists a
  ``tangential'' plane \mbox{$\Pi_T'=\Pi_T\cap\Pi'$} at $p_X$, such that
  all points of \mbox{$S'\!\setminus\!\{p_X\}$} are on one side of
  $\Pi_T'$. Hence, $p_X$ is extremal.

  If $\Conv(X)$ is not a \mbox{$(d\!-\!3)$-dimensional} facet of
  $\CH(S)$, then all \mbox{$(d\!-\!1)$-dimensional} hyperplanes
  containing $X$ have points of $S$ on both sides, and therefore $p_x$
  is not extremal in $S'$. Assume the contrary: at least one
  \mbox{$(d\!-\!1)$-dimensional} hyperplane, $\Pi_T$, containing $X$
  exists, such that all points of \mbox{$S\!\setminus\! X$} are on one
  side of $\Pi_T$. Then we could tilt $\Pi_T$ keeping all of its
  contained points and consuming the ones it hits while tilting, until
  $\Pi_T$ contains $d$ points; i.e., until $\Pi_T$ consumed two more
  points, $q_1$ and $q_2$. Still all points of $S$, except the ones
  contained in $\Pi_T$, are on one side of $\Pi_T$. Observe that a
  hyperplane spanned by $d$ points (in a point set in general
  position) is a \mbox{$(d\!-\!1)$-dimensional} hyperplane. Hence,
  $\Pi_T$ has become a supporting hyperplane of a
  \mbox{$(d\!-\!1)$-dimensional} facet, $\Conv(X\cup\{q_1,q_2\})$, of
  $\CH(S)$. As the convex hull of every subset of $(X\cup\{q_1,q_2\})$
  is a facet of $\CH(S)$, this is a contradiction to the assumption
  that $\Conv(X)$ is not a \mbox{$(d\!-\!3)$-dimensional} facet of
  $\CH(S)$.
\end{proof}

\section{Higher Dimensional Versions of The Order and Discrepancy
  Lemmas}\label{sec:higher}

We prove the higher dimensional versions of the Order and Discrepancy
Lemmas from~\cite{triangmonojournal}. The proofs are essentially the
same as in the planar case, with the difference that some facts we
used in the plane are now provided by the lemmas in the previous
sections.

Recall that in a partial order a \emph{chain} is a set of pairwise
comparable elements, whereas an \emph{antichain} is a set of pairwise
incomparable elements.

\subsection{Order Lemma}

\begin{lemma}\label{lem:ordergen} 
  Let $S$ be a set of $\eta+d+1$ points $(\eta\geq0)$ in general
  position in $\mathbb{R}^d$ $(d\geq2)$, such that $\Conv(S)$ is a
  $d$-simplex.
  Then there exists a triangulation of $S$, such that at least
  $(d-1)\eta + \eta^{(2^{(1-d)})}+1$ of its $d$-simplices contain a
  convex hull point of~$S$.
\end{lemma}

\begin{proof}
  Let $I$ be the set of the $\eta$ interior points of $S$.
  Let $\mathcal{F} = \{F_1, \dots, F_{d+1}\}$ be the set of the
  $(d-1)$-dimensional faces of $\CH(S)$.
  For each $F_i \in \mathcal{F}$ we define a partial order
  $\leq_{F_i}$ on $I$. We say that $p \leq_{F_i} q$ $(p,q\in I)$ if
  $p$ is in the interior of the $d$-simplex $\Conv(F_i \cup \{q\})$.
  Our goal is to obtain a ``long'' chain $C^*$ with respect to some
  $F^*\in\mathcal{F}$ such that $|C^*| \geq \eta^{(2^{(1-d)})}$.

  By Dilworth's Theorem~\cite{dilworth} w.r.t. $\leq_{F_{d+1}}$, there
  exists a chain or an antichain $C_{d+1}$ in $I$ of size at least
  $\sqrt{\eta}\geq \eta^{(2^{(1-d)})}$. If $C_{d+1}$ is a chain then
  we obtain $C^*=C_{d+1}$, $|C^*|\geq\eta^{(2^{(1-d)})}$, and
  $F^*=F_{d+1}$.
  Otherwise, we iteratively apply Dilworth's Theorem
  w.r.t. $\leq_{F_i}$ to the points of the antichain $C_{i+1}$, $i$
  from $d$ downto $3$, to obtain a chain or antichain $C_{i}$ of size
  at least $\sqrt{|C_{i+1}|} = \eta^{(2^{(i-d-2)})}$.
  As soon as $C_{i}$ is a chain, terminate with $C^*=C_{i}$,
  $F^*=F_{i}$, and $|C^*| \geq \eta^{(2^{(1-d)})}$.
  Otherwise, the process ends with the antichain $C_3$ of size at
  least $\eta^{(2^{(1-d)})}$. But, similar to the planar case, an
  antichain with respect to all but two faces is a chain with respect
  to the remaining two faces. Hence, $C^* = C_3$, $F^*=F_2$, with
  $|C^*|\geq \eta^{(2^{(1-d)})}$.

  Let $p_1\leq_{F^*}\ldots\leq_{F^*} p_r$ $(r=|C^*|)$ be the points of
  $C^*$. Construct a triangulation $\mathcal{T}$ of $S$, starting with
  $\mathcal{T}$ consisting only of the $d$-simplex $\Conv(S)$. Then
  insert the points of $C^*$ into $\mathcal{T}$ in the order
  $p_r,\ldots, p_1$. With each step one $d$-simplex is replaced by
  $(d+1)$ new ones. This results in an intermediate triangulation
  $\mathcal{T}$ of $((S\cap\CH(S))\cup \{p_1\ldots p_r\})$ consisting
  of $(dr+1)$ many $d$-simplices, each of which having at least one
  point in $\CH(S)$ as a vertex.

  Let $\sigma_i$, $1\leq i \leq dr+1$, be the $d$-simplices of
  $\mathcal{T}$, let $\eta_i$ be the number of interior points
  of $\sigma_i$, and let $p_i$ be a vertex of $\sigma_i$ that is also
  in $\CH(S)$.
  By Lemma~\ref{lem:starsimplex3d} there exists a triangulation of
  $S\cap\sigma_i$ such that $(d-1)(\eta_i+d+1)-d^2+2$ of its
  $d$-simplices have $p_i$ as a vertex.
  Therefore, the remaining points can be inserted into $\mathcal{T}$,
  such that at least
  $\sum_{i=1}^{dr+1}\left((d-1)\eta_i+1\right) =
  (d-1)(\eta-r) + (dr+1) = (d-1)\eta + r + 1$ of the $d$-simplices
  of $\mathcal{T}$ have at least one point in $\CH(S)$. Since $r\geq
  \eta^{(2^{(1-d)})}$, at least $(d-1)\eta + \eta^{(2^{(1-d)})} + 1$
  many $d$-simplices have at least one point in $\CH(S)$.
\end{proof}

We are now able to prove the high-dimensional variation of the ``Order
Lemma'':

\begin{lemma}[Generalized Order Lemma]\label{lem:order}
  Let $S$ be a set of $n \geq d+1$ points in general position in
  $\mathbb{R}^d$ $(d>2)$ with $h=|S\cap\CH(S)|$.
  Then there exists a triangulation of $S$, such that at least
  $(d-1)n+(n-h)^{(2^{(1-d)})}+2h-c_d$ of its $d$-simplices have at
  least one point in $\CH(S)$, with $c_d$ as defined in
  Lemma~\ref{lem:convtriang}.
\end{lemma}

\begin{proof}
  Let $S'=S\cap\CH(S)$ be the set of convex hull points of $S$.
  If $h>d(d+1)$, then by Lemma~\ref{lem:convtriang} there exists a
  triangulation of $S'$ of size $\tau \geq (d+1)h-c_d$.
  If $h \leq d(d+1)$, then by Theorem~\ref{thm:n-d} any triangulation
  of $S'$ has size at least
  $\tau \geq h-d = (d+1)h -dh-d \geq (d+1)h-c_d$.

  Let $\sigma_i$, $1\leq i\leq\tau$, be the $d$-simplices of the
  triangulation of $S'$, and let $\eta_i$ be the number of interior
  points of $\sigma_i$.
  By Lemma~\ref{lem:ordergen} there exists a triangulation
  $\mathcal{T}_i$ of $S\cap\sigma_i$, such that at least $(d-1)\eta_i
  + \eta_i^{(2^{(1-d)})} + 1$ of the $d$-simplices of $\mathcal{T}_i$
  have at least one point in $\CH(S\cap\sigma_i)$.
  In total we obtain a triangulation $\mathcal{T}$ of $S$, such that
  at least
  $\sum_{i=1}^{\tau}\left( (d-1)\eta_i + \eta_i^{(2^{(1-d)})} + 1
  \right) \geq
  (d-1)\sum_{i=1}^{\tau}\left(\eta_i\right) + \left(\sum_{i=1}^{\tau}
    \eta_i\right)^{(2^{(1-d)})} + \tau \geq
  (d-1)(n-h) + (n-h)^{(2^{(1-d)})} + (d+1)h-c_d =
  (d-1)n +(n-h)^{(2^{(1-d)})}+ 2h-c_d$ of the $d$-simplices of
  $\mathcal{T}$ have at least one point in $\CH(S)$.
\end{proof}

\subsection{Discrepancy Lemma}

Let $S$ be a $k$-colored set of $n$ points in general position in
$\mathbb{R}^d$ and let $S_1,S_2, \dots, S_k$ be its color classes.
Recall that we consider $k$ and $d$ to be constants w.r.t. $n$,
i.e., $k$ and $d$ are independent of~$n$.
We define the \emph{discrepancy} $\delta(S)$ of $S$ to be the sum of
differences between the sizes of its biggest chromatic class and the
remaining classes. Let $S_{\max}$ be the chromatic class with the
maximum number of elements. Then $\delta(S)=\sum
(|S_{\max}|-|S_i|)=(k-1)|S_{\max}|-|S\setminus{S_{\max}}| =
k|S_{\max}|-n$.
Further, we denote with $S_{\min}$ the chromatic class with the least
number of elements.

We start with two statements describing the interaction of
$\delta(S)$, $S_{\max}$, and $S_{\min}$.

\begin{lemma}\label{lem:dS_min_max}
  Let $S$ be a $k$-colored set of $n$ points in general position in
  $\mathbb{R}^d$.
  Let $f_{(n,d,k)}$ be some function on $k$, $d$, and $n$.
  If $|S_{\min}| \leq \frac{n}{k}-(k-1)\cdot f_{(n,d,k)}$ then
  $|S_{\max}| \geq \frac{n}{k} + f_{(n,d,k)}$, and $\delta(S) \geq
  k\cdot f_{(n,d,k)}$.
\end{lemma}

\begin{proof}
  From $|S_{\min}| \leq \frac{n}{k}-(k-1)\cdot f_{(n,d,k)}$ we get
  $|S\setminus S_{\min}| = n-|S_{\min}| \geq
  n-\frac{n}{k}+(k-1)\cdot f_{(n,d,k)} =
  (k-1)\cdot\left(\frac{n}{k}+ f_{(n,d,k)}\right)$.
  As there exist $(k-1)$ color classes besides $|S_{\min}|$, all not
  bigger than $|S_{\max}|$, we have $|S_{\max}| \geq \frac{n}{k} +
  f_{(n,d,k)}$.
  This leads to $\delta(S)= k|S_{\max}|-n \geq k\left(\frac{n}{k}
    + f_{(n,d,k)}\right)-n = k\cdot f_{(n,d,k)}$.
\end{proof}

The following corollary is a direct consequence of
Lemma~\ref{lem:dS_min_max}.

\begin{corollary}\label{cor:dS_min_max}
  Let $S$ be a $k$-colored set of $n$ points in general position in
  $\mathbb{R}^d$.
  Let $f_{(n,d,k)}$ be some function on $k$, $d$, and $n$.
  If $\delta(S) < k\cdot f_{(n,d,k)}$ then $|S_{\min}| >
  \frac{n}{k}-(k-1)\cdot f_{(n,d,k)}$.
\end{corollary}

The previous two technical statements will be needed for the
Theorems~\ref{thm:triang_or_dis_d_base}
and~\ref{thm:triang_or_dis_d}.
For the sake of completeness we state the ``original'' Discrepancy
Lemma for $d=k=2$ from~\cite{triangmonojournal}.

\begin{lemma}[Discrepancy Lemma~\cite{triangmonojournal}]\label{lem:discrepancyd2k2}
  Let $S$ be a $2$-colored set of $n\geq3$ points in general position
  in $\mathbb{R}^2$, such that $\delta(S)\geq2$.
  Then $S$ determines at least $\frac{\delta(S)-2}{6}(n+\delta(S))$
  empty monochromatic triangles.
\end{lemma}

In the following we proof the high-dimensional variation of this
``Discrepancy Lemma'':

\begin{lemma}[Generalized Discrepancy Lemma]\label{lem:discrgen}
  Let $S$ be a $k$-colored set of $n>k\cdot 4^{d^2(d+1)}$ points in
  general position in $\mathbb{R}^d$, with $d\geq k > 3$.
  Then $S$ determines $\Omega{\left(n^{d-k+1}\cdot(\delta(S)+\log
    n)\right)}$ empty monochromatic $d$-simplices.
\end{lemma}

\begin{proof}
  Let $S_{\max}$ be the largest chromatic class of $S$.
  Consider a subset $X$ of $d-k+1$ points of $S_{\max}$.
  From the requirements of the lemma we have $d>3$, $|S_{\max}|\geq
  \left\lceil\frac{n}{k}\right\rceil > 4^{d^2(d+1)}$, and $1\leq
  |X|\leq d-3$.
  Thus we may apply Lemma~\ref{lem:starset} to $X$ which guarantees
  the existence of a $d$-simplicial complex $\mathcal{K}_X$ with
  vertex set $S_{\max}$, such that $\mathcal{K}_X$ has size at least
  $(d-(d-k+1))|S_{\max}|+\frac{\log_2{|S_{\max}|}}{2(d-(d-k+1))}-2c_{d-1}=
  (k-1)|S_{\max}|+\frac{\log_2{|S_{\max}|}}{2(k-1)}-2c_{d-1}$ and all
  $d$-simplices of $\mathcal{K}_X$ have $X$ in their vertex set.
  Since every point of $S \setminus S_{\max}$ is in at most one
  $d$-simplex of $\mathcal{K}_X$, $\mathcal{K}_X$ contains at least
  $\delta(S)+\frac{\log_2{|S_{\max}|}}{2(k-1)}-2c_{d-1}$ empty
  monochromatic $d$-simplices.

  We do this counting for each of the $\binom{|S_{\max}|}{d-k+1}$
  subsets of $(d-k+1)$ points of $S_{\max}$, and over-count each empty
  monochromatic $d$-simplex at most $\binom{d+1}{d-k+1}$ times.
  Hence, in total we get
  $\frac{\binom{|S_{\max}|}{d-k+1}}{\binom{d+1}{d-k+1}} \cdot \left(
    \delta(S)+\frac{\log_2{|S_{\max}|}}{2(k-1)}-2c_{d-1} \right)$ empty
  monochromatic $d$-simplices.
  As $|S_{\max}|\geq \left\lceil\frac{n}{k}\right\rceil$, and $d$,
  $c_{d-1}$ (see Lemma~\ref{lem:convtriang}), and $k$ are constant
  w.r.t. $n$, we get $\Omega{\left(n^{d-k+1}\cdot (\delta(S)+\log n)
    \right)}$ empty monochromatic $d$-simplices in $S$.
\end{proof}

Observe, that this ``Generalized Discrepancy Lemma'' is not applicable
for small values of $k$ and $d$. With the 2-colored variant in
$\mathbb{R}^2$ already provided in Lemma~\ref{lem:discrepancyd2k2}
(\cite{triangmonojournal}), we generalize it to $\mathbb{R}^d$ in the
next lemma.

\begin{lemma}\label{lem:discr2colors}
  Let $S$ be a $2$-colored set of $n>2d$ points in general position in
  $\mathbb{R}^d$, with $d\geq3$.
  Then $S$ determines $\Omega{\left(n^{d-1}\cdot\delta(S)\right)}$
  empty monochromatic $d$-simplices.
\end{lemma}

\begin{proof}
  Let $S_{\max}$ be the largest chromatic class of $S$.
  Consider a subset $X$ of $d-1$ points of $S_{\max}$.
  From the requirements of the lemma we have $d\geq3$, $|S_{\max}|\geq
  \left\lceil\frac{n}{2}\right\rceil > d$, and $|X|= d-1$.
  Thus we may apply Lemma~\ref{lem:starset2r} to $X$ which guarantees
  the existence of a $d$-simplicial complex $\mathcal{K}_X$ with
  vertex set $S_{\max}$, such that $\mathcal{K}_X$ has size at least
  $|S_{\max}| - d$ and all $d$-simplices of $\mathcal{K}_X$ have $X$
  in their vertex set.
  Since every point of $S \setminus S_{\max}$ is in at most one
  $d$-simplex of $\mathcal{K}_X$, $\mathcal{K}_X$ contains at least
  $\delta(S)-d$ empty monochromatic $d$-simplices.

  We do this counting for each of the $\binom{|S_{\max}|}{d-1}$
  subsets of $(d-1)$ points of $S_{\max}$, and over-count each empty
  monochromatic $d$-simplex at most $\binom{d+1}{d-1}=\binom{d+1}{2}$
  times.
  Hence, in total we get
  $\frac{\binom{|S_{\max}|}{d-1}}{\binom{d+1}{2}} \cdot \delta(S)$
  empty monochromatic $d$-simplices.
  As $|S_{\max}|\geq \left\lceil\frac{n}{2}\right\rceil$, and $d$ is
  constant w.r.t. $n$, we get $\Omega{\left(n^{d-1}\cdot \delta(S)
    \right)}$ empty monochromatic $d$-simplices in $S$.
\end{proof}

The still missing $3$-colored case of the ``Discrepancy Lemma'' turns
out to be quite difficult. In the remaining three lemmas of this
section we will first prove the variant for $\mathbb{R}^3$, then give
a general bound for $\mathbb{R}^d$ and $d > 4$, and lastly providing
the missing case of $\mathbb{R}^4$.

\begin{lemma}\label{lem:discr3d} 
  Let $S$ be a $3$-colored set of $n\geq 12$ points in general
  position in $\mathbb{R}^3$. Then $S$ determines at least
  $\frac{\delta(S)-10}{12}\cdot n +3$ empty monochromatic
  $3$-simplices.
\end{lemma}

\begin{proof}
  Let $S_{\max}$ be the largest chromatic class of $S$.
  Let $p$ be a point of $S_{\max}$.
  From the requirements of the lemma we have $d=3$ and $|S_{\max}|\geq
  \left\lceil\frac{n}{3}\right\rceil \geq 4$.
  Thus we may apply Lemma~\ref{lem:starp3d} to $p$ which guarantees
  the existence of a $3$-simplicial complex $\mathcal{K}_p$ with
  vertex set $S_{\max}$, such that all $3$-simplices of
  $\mathcal{K}_p$ have $p$ as a vertex, and $\mathcal{K}_p$ has size
  at least
  \begin{itemize}
  \item $2|S_{\max}| - 6$ if $p$ is an interior point of $S_{\max}$
    and
  \item $2|S_{\max}| - \varrho(p)-4$ if $p$ is a convex hull point of
    $S_{\max}$ and $\varrho(p)$ is the degree of $p$ in the 1-skeleton
    of $\CH{(S_{\max})}$.
  \end{itemize}
  Since every point of $S \setminus S_{\max}$ is in at most one
  $3$-simplex of $\mathcal{K}_p$, $\mathcal{K}_p$ contains at least
  $\delta(S)-6$ empty monochromatic $d$-simplices if $p$ is an
  interior point of $S_{\max}$, and $\delta(S)-\varrho(p)-4$ empty
  monochromatic $d$-simplices if $p$ is a convex hull point of
  $S_{\max}$.

  We do this counting for each point in $S_{\max}$, and over-count each
  empty monochromatic $3$-simplex at most $4$ times.
  Denote with $h$ the number of convex hull points of $S_{\max}$. We
  know from Theorem~\ref{thm:lowerbound} that summing over all convex
  hull points of $S_{\max}$ we have
  $\sum{\varrho(p)}=2\cdot(3h-6)=6h-12$.
  Hence, in total we get
  $\frac{1}{4}\cdot \left( (\delta(S)-6)\cdot(|S_{\max}|-h)+
    (\delta(S)-4)\cdot h - \sum{\varrho(p)} \right) =
  \frac{1}{4}\cdot \left( (\delta(S)-6)\cdot|S_{\max}| - 4h + 12
  \right) \geq \frac{\delta(S)-10}{4}\cdot|S_{\max}| + 3$ empty
  monochromatic $3$-simplices.
  As $|S_{\max}|\geq \left\lceil\frac{n}{3}\right\rceil$, we get at
  least $\frac{\delta(S)-10}{12}\cdot n + 3$ empty monochromatic
  $3$-simplices in $S$.
\end{proof}

\begin{lemma}\label{lem:discr_k3_d5+}
  Let $S$ be a 3-colored set of $n>3d+15$ points in general
  position in $\mathbb{R}^d$ $(d>4)$.
  Then $S$ determines $\Omega(n^{d-2}\cdot\delta(S))$ empty
  monochromatic $d$-simplices.
\end{lemma}

\begin{proof}
  Let $S_{\max}$ be the largest chromatic class of $S$.
  Consider a subset $X$ of \mbox{$d\!-\!2$} points of~$S_{\max}$. Note
  that $|S_{\max}|\geq \left\lceil\frac{n}{3}\right\rceil$.
  Denote with $\Pi$ the \mbox{$(d\!-\!3)$-dimensional} hyperplane
  containing $X$ and with $\Pi'$ a $3$-dimensional hyperplane
  orthogonal to $\Pi$.
  Project $S_{\max}$ orthogonally to $\Pi'$, and let $S_{\max}'$ be
  the resulting image. The set $X$ is projected to a single point
  $p_X$ in $\Pi'$.

  By Lemma~\ref{lem:extremeprojection}, $p_X$ is an extremal point of
  $S_{\max}'$ only if $\Conv(X)$ is a \mbox{$(d\!-\!3)$-dimensional}
  facet of $\CH(S_{\max})$.
  By the upper bound theorem \cite{upperbound}, the convex hull of a point set in
  $\mathbb{R}^d$ has size at most
  $\Theta(n^{\lfloor\frac{d}{2}\rfloor})$. 
  Obviously, this bound applies to the number of all $\xi$-dimensional
  facets, $1\leq\xi<d$, of $\CH(S_{\max})$, as $d$ is constant; i.e.,
  independent of~$|S_{\max}|$.

  On the other hand, the total number of different subsets of
  \mbox{$d\!-\!2$} points of $S_{\max}$ is $\binom{|S_{\max}|}{d-2}
  \geq \binom{\frac{n}{3}}{d-2} = \Theta(n^{d-2})$.
  As $d-2 > \lfloor\frac{d}{2}\rfloor$ for $d>4$, there exist
  $\Theta(n^{d-2})-\Theta(n^{\lfloor\frac{d}{2}\rfloor}) =
  \Theta(n^{d-2})$ different subsets $X$, such that $p_X$ is an
  interior point of $S_{\max}'$.

  For each such subset $X$ apply Lemma~\ref{lem:starset_d-2}, as
  $|S_{\max}|\geq \left\lceil\frac{n}{3}\right\rceil > d+5$ and $d>3$.
  This guarantees for each $X$ the existence of a $d$-simplicial
  complex $\mathcal{K}_X$ with vertex set $S_{\max}$, such that
  $\mathcal{K}_X$ has size at least $2|S_{\max}|-2d-8$ and all
  $d$-simplices of $\mathcal{K}_X$ have $X$ in their vertex set.
  Since every point of $S \setminus S_{\max}$ is in at most one
  $d$-simplex of $\mathcal{K}_X$, $\mathcal{K}_X$ contains at least
  $\delta(S)-2d-8$ empty monochromatic $d$-simplices.

  As we can do this counting for $\Theta(n^{d-2})$ different subsets,
  and over-count each empty monochromatic $d$-simplex at most
  $\binom{d+1}{d-2}$ times, we get at least $\Theta(n^{d-2})\cdot
  (\delta(S)-2d-8)$ empty monochromatic $d$-simplices in total.
\end{proof}

For $\mathbb{R}^4$ the simple asymptotic counting from the previous
proof does not work. We have to take a more detailed look.

\begin{lemma}\label{lem:discr_k3_d4}
  Let $S$ be a 3-colored set of $n>27$ points in general
  position in $\mathbb{R}^4$.
  Then $S$ determines $\Omega(n^{2}\cdot\delta(S))$ empty
  monochromatic $4$-simplices.
\end{lemma}

\begin{proof}
  Let $S_{\max}$ be the largest chromatic class of $S$. Note that
  $|S_{\max}|\geq \left\lceil\frac{n}{3}\right\rceil$.
  Recall that the size of $\CH(S_{\max})$ is bound by
  $O(|S_{\max}|^{\lfloor\frac{4}{2}\rfloor}) =
  O(|S_{\max}|^{2})$. Thus, there are also at most quadratically many
  edges on $\CH(S_{\max})$.
  We distinguish two cases depending on the number of edges on
  $\CH(S_{\max})$.

  \begin{enumerate}
  \item[1)] If less than quadratically many edges are on
    $\CH(S_{\max})$, then there exist $\Theta(|S_{\max}|^{2})$ many
    edges that are no $1$-dimensional facet of $\CH(S_{\max})$.
    Consider a subset $X$ of $2$ points of~$S_{\max}$, spanning such
    an edge.
    Denote with $\Pi$ the line containing $X$ and with $\Pi'$ a
    $3$-dimensional hyperplane orthogonal to $\Pi$.
    Project $S_{\max}$ orthogonally to $\Pi'$, and let $S_{\max}'$ be
    the resulting image. The set $X$ is projected to a single point
    $p_X$ in $\Pi'$.
    By Lemma~\ref{lem:extremeprojection}, $p_X$ is an interior point
    of $S_{\max}'$, as $\Conv(X)$ is not an edge of $\CH(S_{\max})$.
    Apply Lemma~\ref{lem:starset_d-2} to $X$, as $|S_{\max}|\geq
    \left\lceil\frac{n}{3}\right\rceil > d+5$ and $d=4>3$.
    This guarantees for $X$ the existence of a $4$-simplicial complex
    $\mathcal{K}_X$ with vertex set $S_{\max}$, such that
    $\mathcal{K}_X$ has size at least $2|S_{\max}|-16$ and all
    $4$-simplices of $\mathcal{K}_X$ have $X$ in their vertex set.
    Since every point of $S \setminus S_{\max}$ is in at most one
    $4$-simplex of $\mathcal{K}_X$, $\mathcal{K}_X$ contains at least
    $\delta(S)-16$ empty monochromatic $4$-simplices.
    As we can do this counting for $\Theta(|S_{\max}|^{2}) =
    \Theta(n^{2})$ different subsets, and over-count each empty
    monochromatic $4$-simplex at most $\binom{5}{2}$ times, we get at
    least $\Theta(n^{2})\cdot (\delta(S)-16)$ empty monochromatic
    $4$-simplices in total.

  \item[2)] If there are $\Theta(|S_{\max}|^2)$ many edges on
    $\CH(S_{\max})$, then there are also $\Theta(|S_{\max}|^2)$ many
    tetrahedra on $\CH(S_{\max})$ (because the number of tetrahedra is
    at least a sixth of the number of edges), and obviously
    $|S_{\max}\cap\CH(S_{\max})| = \Theta(|S_{\max}|)$.
    For a point $p\in S_{\max}$ make a pulling triangulation
    $\mathcal{K}_p$ of
    $(S_{\max}\cap\CH(S_{\max}))\cup\{p\}$. Inserting the remaining
    points of $S_{\max}$ into $\mathcal{K}_p$ does not decrease the
    number of 4-simplices in $\mathcal{K}_p$, which have $p$ as a
    vertex. Remove all 4-simplices from $\mathcal{K}_p$ that don't
    have $p$ as a vertex. Then $\mathcal{K}_p$ is a 4-dimensional
    simplicial complex, such that every 4-simplex has $p$ as a vertex
    and $\mathcal{K}_p$ is of size
    \begin{enumerate}
    \item[a)] $\Theta(|S_{\max}|^2)$, if $p$ is an interior point of
      $S_{\max}$, or
    \item[b)] $\Theta(|S_{\max}|^2)-\varrho(p)$, if $p$ is an extremal
      point of $S_{\max}$, where $\varrho(p)$ is the number of
      tetrahedra in $CH(S_{\max})$, having $p$ as a vertex.
    \end{enumerate}
    For case b) observe, that
    $\sum_{p\in(S_{\max}\cap\CH(S_{\max}))}{\varrho(p)} =
    4\cdot\Theta(|S_{\max}|^2)$. Thus on average, at least
    $\Omega(|S_{\max}|)$ points of $S_{\max}$ have at most
    $O(|S_{\max}|)$ incident tetrahedra in $CH(S_{\max})$. Hence,
    $\Theta(|S_{\max}|^2)-\varrho(p) = \Theta(|S_{\max}|^2)$ for
    $\Theta(|S_{\max}|)$ points $p\in S_{\max}$.

    All 4-simplices of $\mathcal{K}_p$ are empty of points of
    $S_{\max}$ by construction. Since every point of $S \setminus
    S_{\max}$ is in at most one $4$-simplex of $\mathcal{K}_p$,
    $\mathcal{K}_p$ contains at least $\Theta(|S_{\max}|^2) -
    2|S_{\max}|+\delta(S)-2d-8 = \Theta(n^2)$ empty monochromatic
    $4$-simplices. Note that $\delta(S)=O(n)$.
    As we can do this counting for $\Theta(|S_{\max}|) = \Theta(n)$
    different points, and over-count each empty monochromatic
    $4$-simplex at most $5$ times, we get $\Omega(n^{3})\geq
    \Omega(n^{2}\cdot\delta(S))$ empty monochromatic $4$-simplices in
    total.
  \end{enumerate}\vspace{-8mm}
\end{proof}

With this last lemma in a line of five lemmas in total and
including~\cite{triangmonojournal}, we now have a ``Discrepancy
Lemma'' type of statement for all $k$-colored point sets in
$\mathbb{R}^d$, for every combination of $d\geq2$ and $2\leq k\leq d$.

\section{Empty Monochromatic Simplices in $k$-Colored Point Sets}
\label{sec:ems}

In this section we present our results on the minimum number of empty
monochromatic $d$-sim\-plices determined by any $k$-colored set of $n$
points in general position in $\mathbb{R}^d$.
Some first bounds follow directly from the results in the previous
section.

\begin{theorem}\label{thm:(d+1)}
  Every \mbox{$(d\!+\!1)$-colored} set $S$ of $n \geq
  (d+1)\cdot4^{d(c_d+1)}$ points in general position in $\mathbb{R}^d$
  $(d>2)$, $c_d$ defined as in Lemma~\ref{lem:convtriang}, determines
  an empty monochromatic $d$-simplex.
\end{theorem}

\begin{proof}
  Let $S_{\max}$ be the largest chromatic class of $S$.
  From the requirements of the theorem we have $d>2$ and $|S_{\max}|\geq
  \left\lceil\frac{n}{d+1}\right\rceil \geq 4^{d(c_d+1)} =
  4^{d^4+d^3+d^2+d} > 4^{d^2(d+1)}$.
  By Theorem~\ref{thm:dn+log}, $S_{\max}$ has a triangulation
  $\mathcal{T}$ of size at least
  $d|S_{\max}|+\frac{\log_2{|S_{\max}|}}{2d}-c_d$. All the
  $d$-simplices of $\mathcal{T}$ are of the same color and empty of
  points of $S_{\max}$. There are at most $d|S_{\max}|$ points in $S$
  of the remaining colors, and each of these points is in at most one
  $d$-simplex of $\mathcal{T}$.

  Therefore, at least $\frac{\log_2{|S_{\max}|}}{2d}-c_d \geq
  \frac{2d(c_d+1)}{2d}-c_d=1$ of the $d$-simplices of $\mathcal{T}$ are
  empty of points of $S$.
\end{proof}

Note that \mbox{$d>2$} is crucial here, as for \mbox{$d=2$} Devillers
et al.~\cite{chromaticvariants} showed that there are arbitrarily
large 3-colored sets which do not contain an empty monochromatic
triangle.

As an immediate corollary of Theorem~\ref{thm:(d+1)} we have:

\begin{corollary}\label{cor:(d+1)}
  Every \mbox{$(d\!+\!1)$-colored} set $S$ of $n \geq
  (d+1)\cdot4^{d(c_d+1)}$ points in general position in $\mathbb{R}^d$
  $(d>2)$, $c_d$ defined as in Lemma~\ref{lem:convtriang}, determines
  at least a linear number of empty monochromatic $d$-simplices.
\end{corollary}

\begin{proof}
  By Theorem~\ref{thm:(d+1)} there exists a constant $\mu_d\leq
  (d+1)\cdot4^{d(c_d+1)}$ such that every subset of $S$ of
  $\mu_d$ points determines at least one empty monochromatic
  $d$-simplex.
  Divide $S$ (with parallel \mbox{$(d\!-\!1)$-dimensional} hyperplanes) into
  $\left\lfloor\frac{n}{\mu_d}\right\rfloor$ subsets of $\mu_d$ points
  each.
  Hence, in total there exist at least
  $\left\lfloor\frac{n}{\mu_d}\right\rfloor$ empty monochromatic
  $d$-simplices in~$S$.
\end{proof}

The next result follows immediately from Lemma~\ref{lem:discrgen} and
provides a first general lower bound.

\begin{corollary}\label{cor:ems_in_dd+1}
  Let $S$ be a $k$-colored set of $n>k\cdot 4^{d^2(d+1)}$ points in
  general position in $\mathbb{R}^d$, with $d\geq k > 3$. Then $S$
  determines $\Omega{(n^{d-k+1} \log{n})}$ empty monochromatic
  $d$-simplices.
\end{corollary}

\begin{proof}
  This is a direct consequence of Lemma~\ref{lem:discrgen} since every
  colored set has discrepancy at least~$0$.
\end{proof}

We will further improve on this result in Theorem~\ref{thm:ems_in_d3+}
below. The next theorem is central for this improvement and provides a
relation between the number of empty monochromatic $d$-simplices of an
arbitrary color in a $d$-colored point set $S\subset\mathbb{R}^d$, and
convex subsets of $S$ with high discrepancy.

\begin{theorem}\label{thm:triang_or_dis_d_base}
  Let $S$ be a $d$-colored set of $n\geq 3d\cdot(2c_d)^{(2^{d-1})}$
  points in general position in $\mathbb{R}^d$, $d>2$ and $c_d$ as
  defined in Lemma~\ref{lem:convtriang}.
  For every $1 \leq j \leq d$, either there are $\Omega(n^{1+2^{-d}})$
  empty monochromatic $d$-simplices of color $j$, or there is a convex
  set $C$ in $\mathbb{R}^d$, such that $|S\cap C|=\Theta(n)$ and
  $\delta(S\cap C) = \Omega(n^{(2^{-d})})$.
\end{theorem}

\begin{figure}[htb]
  \centering
  \includegraphics{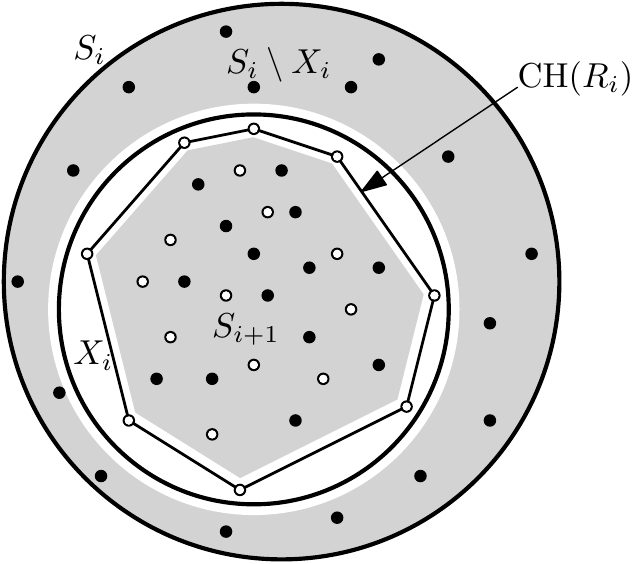}
  \caption{2D-sketch to illustrate the nomenclature of the proof of
    Theorem~\ref{thm:triang_or_dis_d_base}. White points are points of
    chromatic class $j$, black points are points of the other color
    classes.}
  \label{fig:T22_sketch}
\end{figure}

\begin{proof}
  The general idea for the proof is to iteratively peel convex layers
  of color~$j$ from the point set. For each layer we use the
  Generalized Order Lemma to obtain roughly $n^{(2^{1-d})}$ empty
  monochromatic $d$-simplices of color~$j$. If at any moment the
  discrepancy is large enough we terminate the process with the
  desired convex set~$C$. Otherwise, the iteration stops after at most
  $\frac{1}{8}n^{(1-2^{-d})}$ steps.

  Let $S_i$ be the $d$-colored set of points in iteration
  step~$i$. With $S_{i,l}$ we denote the chromatic classes of $S_i$,
  and with $S_{i,max}$ / $S_{i,min}$ we denote the largest / smallest
  chromatic class of $S_i$, respectively. Note that a point of $S_i$
  can only be in one chromatic class, and that $\bigcup_{l=1}^d
  S_{i,l} = S_i$.
  The iteration starts with $S_1=S$. For $i>1$ smaller sets are
  constructed, such that $S_{i+1} \subset S_i$ and $S_{i+1,l}
  \subseteq S_{i,l}$.
  Let $\tilde{n}=\frac{n}{3d}$. As an invariant through all iterations
  we guarantee \[\mbox{Invariant: }|S_i| \geq (d+1)\tilde{n} \mbox{\ .}\]
  The iteration stops either if a convex set $C$ is found, with
  $|S\cap C| = \Theta{(n)}$ and $\delta(S\cap C) \geq
  \frac{\tilde{n}^{(2^{-d})}}{(d-1)}$, or after at most
  $\frac{1}{8}n^{(1-2^{-d})}$ steps.

  Consider the $i$-th step of the iteration.
  We will prove inequalities on the sizes of different subsets, their
  discrepancy, and the size of chromatic classes.
  With $R_i$ we denote the $j$-th chromatic class in step~$i$,
  i.e., $S_{i,j}$ of $S_i$. Further, let $h_i$ be the number of points
  in $\CH{(R_i)}$ and let $X_i=S_i\cap\Conv{(R_i)}$, such that the
  chromatic classes of $X_i$ are $X_{i,l}=S_{i,l}\cap\Conv{(R_i)}$,
  with $X_{i,max}$ / $X_{i,min}$ being the largest / smallest
  chromatic class of $X_i$, respectively. See also
  Figure~\ref{fig:T22_sketch} for an illustration of the different
  sets.

  \begin{enumerate}
  \item[(1)] $\delta(S_i) < \frac{\tilde{n}^{(2^{-d})}}{d-1}$. \\
    If $\delta(S_i) \geq \frac{\tilde{n}^{(2^{-d})}}{d-1}$, then the
    iteration terminates with $C = \Conv(S_i)$, as $S\cap C = S_i$ and
    $|S_i|=\Theta(n)$ by the invariant.
  \item[(2)] $|R_i| > \frac{|S_i|}{d}-\frac{\tilde{n}^{(2^{-d})}}{d} > \tilde{n}$. \\
    By inequality~(1), $\delta(S_i) < \frac{\tilde{n}^{(2^{-d})}}{d-1}=
    d\cdot\frac{\tilde{n}^{(2^{-d})}}{d(d-1)}$. Applying
    Corollary~\ref{cor:dS_min_max} we get $|S_{i,min}| >
    \frac{|S_i|}{d}- (d-1)\cdot\frac{\tilde{n}^{(2^{-d})}}{d(d-1)} =
    \frac{|S_i|}{d}-\frac{\tilde{n}^{(2^{-d})}}{d} \geq
    \frac{(d+1)\tilde{n}-\tilde{n}^{(2^{-d})}}{d} > \tilde{n}$.
    Obviously $|R_i|\geq|S_{i,min}|$, which proves the inequality.
  \item[(3)] $\delta(X_i) < \frac{\tilde{n}^{(2^{-d})}}{d-1}$. \\
    Obviously, $|X_i|\geq|R_i|$. Thus, by inequality~(2), $|X_i|
    >\tilde{n} = \Theta(n)$. Hence, if $\delta(X_i) \geq
    \frac{\tilde{n}^{(2^{-d})}}{d-1}$, then the iteration terminates
    with $C = \Conv(X_i)$.
  \item[(4)] $(d-1)|R_i|-|X_i\setminus R_i| > -\tilde{n}^{(2^{-d})}$. \\
    Assume the contrary: $(d-1)|R_i|-|X_i\setminus R_i| \leq
    -\tilde{n}^{(2^{-d})}$, which can be rewritten to $d|R_i|\leq
    |X_i| - \tilde{n}^{(2^{-d})}$. From inequality~(3) we know that
    $\delta(X_i) < \frac{\tilde{n}^{(2^{-d})}}{d-1}$, which implies
    by Corollary~\ref{cor:dS_min_max} that $|X_{i,min}| >
    \frac{|X_i|}{d}-\frac{\tilde{n}^{(2^{-d})}}{d}$. As obviously
    $|R_i|\geq|X_{i,min}|$, we get $|X_i| - \tilde{n}^{(2^{-d})} \geq
    d|R_i| > |X_i| - \tilde{n}^{(2^{-d})}$, which is a contradiction.
  \item[(5)] $|S_i\setminus X_i| < 2\tilde{n}^{(2^{-d})}$. \\
    Assume the contrary: $|S_i\setminus X_i| \geq
    2\tilde{n}^{(2^{-d})}$. Using inequality~(3) and the definition
    for the discrepancy we get $
    \frac{\tilde{n}^{(2^{-d})}}{d-1}>\delta(X_i) = (d-1)|X_{i,max}| -
    |X_i\setminus X_{i,max}| \geq (d-1)|R_i|-|X_i\setminus
    R_i|$. Further, we know that $|X_i\setminus R_i| = |S_i\setminus
    R_i| - |S_i\setminus X_i|$ and from inequality~(2) we know $|R_i|
    > \frac{|S_i|}{d}-\frac{\tilde{n}^{(2^{-d})}}{d}$. Together with
    the assumption this leads to $\frac{\tilde{n}^{(2^{-d})}}{d-1}>
    (d-1)|R_i|-|S_i\setminus R_i| + |S_i\setminus X_i| = d|R_i|-|S_i|
    + |S_i\setminus X_i| >
    |S_i|-\tilde{n}^{(2^{-d})}-|S_i|+2\tilde{n}^{(2^{-d})}=\tilde{n}^{(2^{-d})}$,
    which is a contradiction.
  \item[(6)] $\delta(S_{i+1}) < \frac{\tilde{n}^{(2^{-d})}}{d-1}$. \\
    Using inequality~(5) we can give the following bound:
    $|X_i|=|S_i|-|S_i\setminus X_i|>
    |S_i|-2\tilde{n}^{(2^{-d})}$.\linebreak From inequality~(3) we get
    $(d-1)|X_{i,max}|-|X_i\setminus X_{i,max}| = \delta(X_i) <
    \frac{\tilde{n}^{(2^{-d})}}{d-1}$ and therefore
    $|R_i|\leq|X_{i,max}|<\frac{|X_i|+\frac{\tilde{n}^{(2^{-d})}}{d-1}}{d}$.
    Combining these inequalities and using the invariant for~$|S_i|$,
    we get $|S_{i+1}|\geq |X_i|-|R_i| >
    |X_i|-\frac{|X_i|+\frac{\tilde{n}^{(2^{-d})}}{d-1}}{d} >
    \frac{d-1}{d}\left(|S_i|-2\tilde{n}^{(2^{-d})}\right) -
    \frac{\tilde{n}^{(2^{-d})}}{d(d-1)} \geq
    \frac{(d-1)(d+1)}{d}\tilde{n} -
    \frac{2(d-1)^2+1}{d(d-1)}\tilde{n}^{(2^{-d})}$. As $d>2$ we may
    evaluate this relation to $|S_{i+1}| > \frac{8}{3}\tilde{n} -
    \frac{3}{2}\tilde{n}^{(2^{-d})} > \tilde{n} = \Theta(n)$.
    Hence, if $\delta(S_{i+1}) \geq \frac{\tilde{n}^{(2^{-d})}}{d-1}$,
    then the iteration terminates with $C = \Conv(S_{i+1})$.
  \item[(7)] $h_i < 2\tilde{n}^{(2^{-d})}$. \\
    As always, assume the contrary: $h_i \geq
    2\tilde{n}^{(2^{-d})}$. We distinguish two cases on whether $R_i$
    is the largest chromatic class of $X_i$ or not.
    \begin{enumerate}
    \item[(a)] If $R_i\neq X_{i,max}$ then $S_{i+1,max}=X_{i,max}$ and
      $|S_{i+1}\setminus S_{i+1,max}| = |X_{i}\setminus X_{i,max}| -
      h_i$. Using inequality~(6) and the definition for the
      discrepancy, we get
      $\frac{\tilde{n}^{(2^{-d})}}{d-1}>\delta(S_{i+1}) =
      (d-1)|S_{i+1,max}| - |S_{i+1}\setminus S_{i+1,max}| =
      (d-1)|X_{i,max}| - |X_{i}\setminus X_{i,max}| + h_i =
      \delta(X_i) + h_i$, which is a contradiction to the assumption,
      as $\delta(X_i)\geq0$.
    \item[(b)] If $R_i= X_{i,max}$, recall that $R_{i+1}$ denotes the
      $j$-th color class of $S_{i+1}$ and observe that $R_{i+1} = R_i
      \setminus (R_i\cap\CH{(R_i)})$. From inequality~(3) and $R_i=
      X_{i,max}$ we derive $\frac{\tilde{n}^{(2^{-d})}}{d-1} >
      \delta(X_i) = d|R_i|-|X_i|= d(|R_{i+1}|+h_i) - (|S_{i+1}|+h_i)$,
      and get $|R_{i+1}|< \frac{\tilde{n}^{(2^{-d})}}{d(d-1)} +
      \frac{|S_{i+1}|+h_i}{d}-h_i =
      \frac{|S_{i+1}|}{d}-(d-1)\cdot\left(\frac{h_i}{d}-
        \frac{\tilde{n}^{(2^{-d})}}{d(d-1)^2}\right)$. As
      $|S_{i+1,min}| \leq |R_{i+1}|$ we get from
      Lemma~\ref{lem:dS_min_max} that $\delta(S_{i+1}) \geq
      d\cdot\left(\frac{h_i}{d}-
        \frac{\tilde{n}^{(2^{-d})}}{d(d-1)^2}\right) = h_i -
      \frac{\tilde{n}^{(2^{-d})}}{(d-1)^2}$. Using inequality~(6) and
      inserting the assumption for $h_i$, results in the contradiction
      $\frac{\tilde{n}^{(2^{-d})}}{d-1}>\delta(S_{i+1}) \geq
      \frac{\tilde{n}^{(2^{-d})}}{d-1}\cdot\left(2(d-1)-\frac{1}{d-1}\right)$,
      as $d>2$.
    \end{enumerate}
  \end{enumerate}

  Using these inequalities we can provide a lower bound on the number
  of empty monochromatic $d$-simplices of color $j$ per step and
  hence, in total, and prove the invariant on $|S_i|$.
  From inequality~(2) we know that $|R_i| > \tilde{n} = \frac{n}{3d}
  \geq d+1$. Thus we may apply the Generalized Order Lemma
  (Lemma~\ref{lem:order}) to $R_i$, which guarantees at least
  $(d-1)|R_i|+(|R_i|-h_i)^{(2^{1-d})}+2h_i-c_d$ interior disjoint
  $d$-simplices of color $j$ with at least one point in $\CH(R_i)$
  each. Only points of \mbox{$(X_i\!\setminus\! R_i)$} can be in these
  $d$-simplices, and each of these \mbox{$|X_i\!\setminus\! R_i|$}
  points lies inside at most one $d$-simplex.
  Therefore, there exist at least $(d-1)|R_i|-|X_i\setminus
  R_i|+(|R_i|-h_i)^{(2^{1-d})}+2h_i-c_d =: \tau_i$ empty monochromatic
  $d$-simplices of color $j$, each of them having at least one point
  in $\CH(R_i)$.
  Using the inequalities~(4), (2), (7), and $h_i \geq 0$, we get
  $\tau_i > -\tilde{n}^{(2^{-d})} +
  \left(\tilde{n}-2\tilde{n}^{(2^{-d})}\right)^{(2^{1-d})} +0 -c_d
  \geq \frac{\tilde{n}^{(2^{1-d})}}{10}$, where the last inequality
  holds for $\tilde{n}\geq(2c_d)^{(2^{d-1})}$.

  The next iteration step $i+1$ considers $S_{i+1}= X_i \setminus
  (R_i\cap\CH{(R_i)})$ and $S_{i+1,l}= S_{i,l}\cap S_{i+1}$, for
  $1\leq l\leq d$. Note that all empty monochromatic $d$-simplices of
  color $j$ from step $i$ have at least one vertex in $\CH{(R_i)}$. As
  the points of $\CH{(R_i)}$ are not in $S_{i+1}$, we do not
  over-count.

  \medskip The iteration either terminates with a convex set $C$,
  such that $|S\cap C|=\Theta(n)$ and $\delta(S\cap C)
  \geq \frac{\tilde{n}^{(2^{-d})}}{d-1} = \Omega(n^{(2^{-d})})$, or it
  ends after $\frac{1}{8}n^{(1-2^{-d})}$ steps.
  With at least $\frac{\tilde{n}^{(2^{1-d})}}{10}$ empty monochromatic
  $d$-simplices of color $j$ per step we get
  $\frac{\tilde{n}^{(2^{1-d})}}{10}\cdot\frac{1}{8}n^{(1-2^{-d})} =
  \frac{1}{80}\cdot\left(\frac{n}{3d}\right)^{(2^{1-d})}\cdot
  n^{(1-2^{-d})} = \Omega(n^{(1+2^{-d})})$ such simplices in total.

  It remains to prove the invariant $|S_i|\geq (d+1)\tilde{n}$. After
  each step we have $S_{i+1}= X_i \setminus (X_i\cap\CH{(R_i)})$ and
  thus $|S_{i+1}| = |S_{i}| - |S_{i}\setminus X_i| - h_i$. With
  inequalities~(5) and~(7) we get $|S_{i+1}| > |S_{i}| -
  2\tilde{n}^{(2^{-d})}-2\tilde{n}^{(2^{-d})} = |S_{i}| -
  4\tilde{n}^{(2^{-d})}$. Therefore, starting with $S_1=S$, there are
  at least $n-4\tilde{n}^{(2^{-d})}\cdot\frac{1}{8}n^{(1-2^{-d})} =
  3d\tilde{n}-\frac{1}{2}\cdot\frac{\tilde{n}^{(2^{-d})}\cdot
    3d\tilde{n}}{(3d\tilde{n})^{(2^{-d})}} =
  \tilde{n}\left(3d-\frac{3d}{2\cdot(3d)^{(2^{-d})}}\right) \geq
  \frac{3d}{2}\tilde{n}>(d+1)\tilde{n}$ points left after
  $\frac{1}{8}n^{(1-2^{-d})}$ steps, as $d>2$.
\end{proof}

We generalize the last result to $k$-colored point sets, for $3\leq
k\leq d$.

\begin{theorem}\label{thm:triang_or_dis_d}
  Let $S$ be a $k$-colored set of $n\geq
  2^{d-k}\left(3k\cdot(2c_d)^{(2^{k-1})}+1\right)$ points in general
  position in $\mathbb{R}^d$, $d>2$ and $c_d$ defined as in
  Lemma~\ref{lem:convtriang}.
  For every $3 \leq k \leq d$ and every $1 \leq j \leq k$, either
  there are $\Omega(n^{d-k+1+2^{-d}})$ empty monochromatic
  $d$-simplices of color $j$, or there is a convex set $C$ in
  $\mathbb{R}^d$, such that $|S\cap C|=\Theta(n)$ and $\delta(S\cap C)
  = \Omega(n^{(2^{-d})})$.
\end{theorem}

\begin{figure}[htb]
  \centering
  \includegraphics{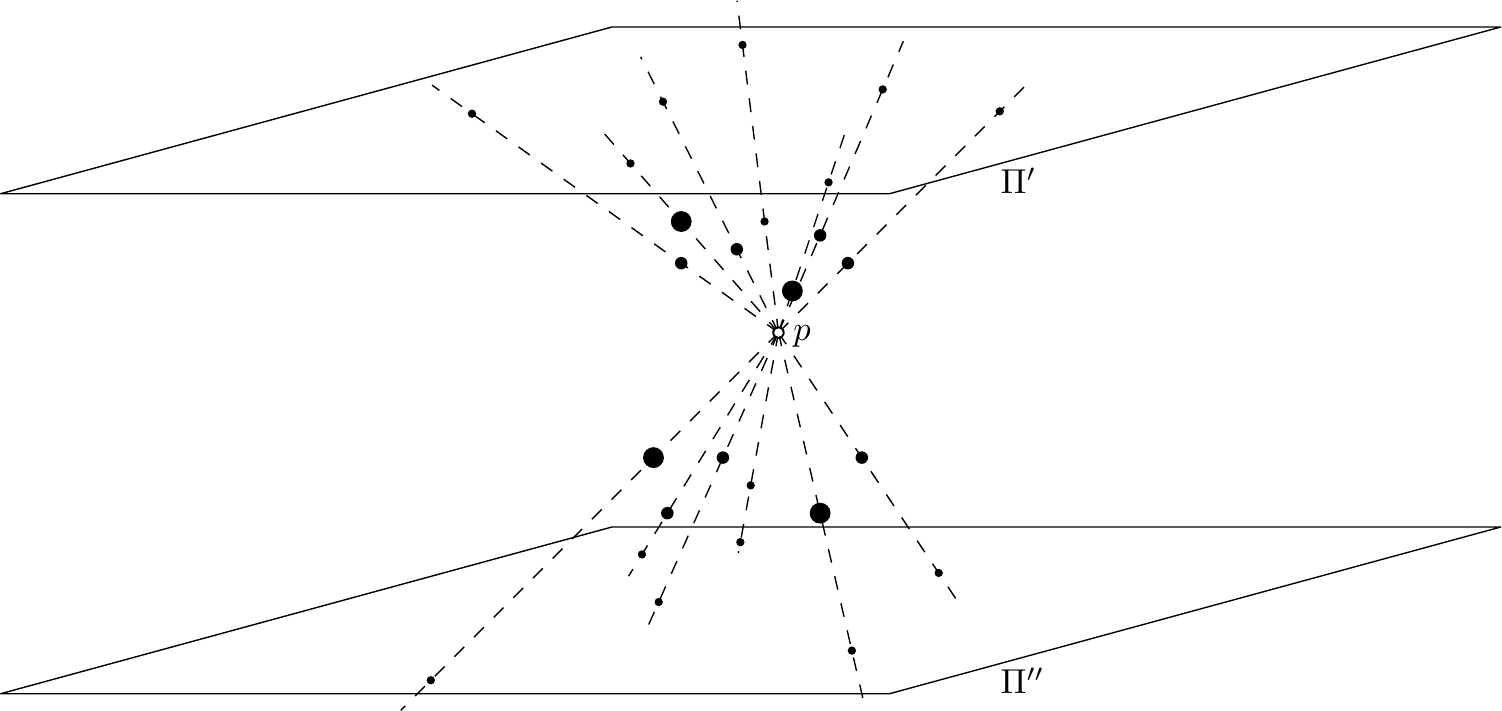}
  \caption{Illustration of the projection, $R^3$ to two
    $2$-dimensional hyperplanes in the sketch, in the proof of
    Theorem~\ref{thm:triang_or_dis_d}..}
  \label{fig:T23_sketch}
\end{figure}

\begin{proof}
  For fixed $k$ we prove the theorem by induction on the dimension,
  and use Theorem~\ref{thm:triang_or_dis_d_base} as an induction base
  for $d=k>2$. Consider the induction step $(d-1)\longrightarrow d$,
  for $d>k$. Denote with $S_j$ the $j$-th, and with $S_{min}$ the
  smallest chromatic class of $S$. If $\delta(S)\geq
  \frac{n^{(2^{-d})}}{k-1}$ then $C=\Conv(S)$ is the desired convex
  set, with $|S\cap C|=\Theta(n)$ and $\delta(S\cap C) =
  \Omega(n^{(2^{-d})})$.
  Thus assume that $\delta(S) < \frac{n^{(2^{-d})}}{k-1} =
  k\cdot\frac{n^{(2^{-d})}}{k(k-1)}$. From
  Corollary~\ref{cor:dS_min_max} we know that
  $|S_j|\geq|S_{min}|>\frac{|S|}{k}-(k-1)\cdot\frac{n^{(2^{-d})}}{k(k-1)}
  = \frac{n-n^{(2^{-d})}}{k}\geq \frac{n}{2k} = \Theta(n)$.

  Let $p\in S_j$ be a point of color $j$.
  For every point $q \in S\setminus\{p\}$ let $r_{q}$ be the infinite
  ray with origin $p$ and passing through $q$.
  Let $\Pi'$ and $\Pi''$ be two $(d-1)$-dimensional hyperplanes
  containing $\Conv(S)$ between them and not parallel to any of the
  rays $r_q$. See Figure~\ref{fig:T23_sketch} for a sketch.
  Project from $p$ every point in \mbox{$S\!\setminus\!\!\{p\}$} to
  $\Pi'$ or $\Pi''$, in the following way. Every ray $r_q$ intersects
  either $\Pi'$ or $\Pi''$ in a point $q'$ or $q''$,
  respectively. Take $q'$ or $q''$ to be the
  projection of $q$ from $p$. Let $S'$ and $S''$ be the sets of these projected
  points in $\Pi'$ and $\Pi''$, respectively.
  The bigger set, assume w.l.o.g. $S'$ in $\Pi'$, is a set of at least
  $\frac{n-1}{2}$ points in general position in $\mathbb{R}^{d-1}$. 

  Apply the induction hypothesis to $S'$ and get either (a)
  $\Omega(n^{d-1-k+1+2^{-d}})$ empty monochromatic
  \mbox{$(d\!-\!1)$}-simplices of color $j$, or (b) a convex set $C$
  in $\mathbb{R}^{d-1}$, such that $|S'\cap C|=\Theta(n)$ and
  $\delta(S'\cap C) = \Omega(n^{(2^{-d+1})})$.

  For case~(b) observe, that the preimage of the point set of a convex
  set in $\Pi'$ is the point set of a convex set in
  $\mathbb{R}^{d}$. Hence, $C$ is a convex set in $\mathbb{R}^{d}$,
  such that $|S\cap C|=\Theta(n)$ and $\delta(S\cap C) =
  \Omega(n^{(2^{-d+1})})$, which trivially implies $\delta(S\cap C) =
  \Omega(n^{(2^{-d})})$.

  For case~(a) note that, if $X$ is the vertex set of an empty
  monochromatic \mbox{$(d\!-\!1)$}-simplex of color $j$ in $\Pi'$,
  then $\Conv(X\cup{p})$ is an empty monochromatic $d$-simplex of
  color $j$ in $\mathbb{R}^{d}$.
  Repeat the projection and the induction for each point $p\in S_j$
  and assume that this always results in case~(a) (because the proof
  is completed if case~(b) happens once).
  This results in a total of
  $\frac{|S_j|}{d+1}\cdot\Omega(n^{d-k+2^{-d}}) =
  \Omega(n^{d-k+1+2^{-d}})$ empty monochromatic $d$-simplices of color
  $j$, as each $d$-simplex gets over-counted at most $(d+1)$ times.
\end{proof}

Combining the last theorem with the ``Generalized Discrepancy Lemma''
(Lemma~\ref{lem:discrgen}) and its different versions for the
3-colored case
(Lemmas~\ref{lem:discr3d}~to~\ref{lem:discr_k3_d4}),
we can prove one of our main results.

\begin{theorem}\label{thm:ems_in_d3+}
  Any $k$-colored set $S$ of $n$ points in general position in
  $\mathbb{R}^d$, $d \geq k \geq 3$, determines
  $\Omega(n^{d-k+1+2^{-d}})$ empty monochromatic $d$-simplices.
\end{theorem}

\begin{proof}
  By Theorem~\ref{thm:triang_or_dis_d} either there exist
  $\Omega(n^{d-k+1+2^{-d}})$ empty monochromatic $d$-simplices, or
  there exists a convex set $C$ in $\mathbb{R}^d$, such that $|S\cap
  C|=\Theta(n)$ and $\delta(S\cap C) = \Omega(n^{(2^{-d})})$.
  In the latter case, there exist $\Omega(n^{d-k+1+2^{-d}})$ empty
  monochromatic $d$-simplices by applying
  Lemma~\ref{lem:discrgen} (for $d \geq k > 3$),
  Lemma~\ref{lem:discr3d} (for $d=k=3$),
  Lemma~\ref{lem:discr_k3_d4} (for $d=4$ and $k=3$), or
  Lemma~\ref{lem:discr_k3_d5+} (for $d>4$ and $k=3$)
  to the point set $(S\cap C)$.
\end{proof}

\section{Empty Monochromatic Simplices in Two Colored Point Sets}\label{sec:ems2c}

For the sake of simplicity, we call the two color classes of a
bi-chromatic point set $S$ ``red'' and ``blue'', and denote these point
sets with $R$ and $B$, respectively.
Observe, that the discrepancy $\delta(S)=(k-1)|S_{max}|-|S\setminus
S_{max}|$ simplifies to $\delta(S)=\big||R|-|B|\big|$ for the
bi-chromatic case $k=2$. This is the same notion of
discrepancy as used in~\cite{triangmonojournal} and~\cite{pachmono}.

Note further, that assuming an upper bound for the discrepancy,
$\delta(S)=\big||R|-|B|\big|<f_n$, for a bi-colored set of $n$ points,
leads to lower and upper bounds for the cardinality of both color
classes in a simple way. The inequality reformulates to $|R|-|B|<f_n$
and $|R|-|B|>-f_n$.  Using $|R|=|S|-|B|$ and $|B|=|S|-|R|$ we make the
following simple observation, which will be used frequently later on.

\begin{observation}\label{obs:ds_bound_2D}
  Let $S$ be a bi-colored set of $n$ points in general position in
  $\mathbb{R}^2$, partitioned into a red point set~$R$ and a blue
  point set~$B$.
  Let $f_n$ be some function on $n$.
  If $\delta(S) < f_n$, then $|B|-f_n<|R|<|B|+f_n$ and
  $|R|-f_n<|B|<|R|+f_n$, and $\frac{n-f_n}{2}<|R|<\frac{n+f_n}{2}$ and
  $\frac{n-f_n}{2}<|B|<\frac{n+f_n}{2}$.
\end{observation}

We adapt the result and proof from~\cite{pachmono} on the number of
empty monochromatic triangles in bi-chromatic point sets to obtain the
central trade off between many empty monochromatic triangles and large
convex sets.

\begin{theorem}\label{thm:newbasecase2d}
  Let $S$ be a bi-colored set of $n$ points in general position in
  $\mathbb{R}^2$, partitioned into a red point set~$R$ and a blue
  point set~$B$.
  Then either there exist $\Omega(n^{4/3})$ empty red triangles, or
  there exists a convex set $C$ in $\mathbb{R}^2$, such that $|S\cap
  C|=\Theta(n)$ and $\delta(S\cap C) = \big||C\cap R|-|C\cap B|\big| =
  \Omega(\sqrt[3]{n})$.
\end{theorem}

\begin{proof}
  Following the lines of~\cite{pachmono} and using their notion, we
  call a point $p\in S$ \emph{rich} if at least
  $\frac{\sqrt[3]{n}}{3}$ empty monochromatic triangles in $S$ have
  $p$ as a vertex.
  The general idea for the proof is to iteratively remove a rich red
  point from the point set. We show that it is possible to find either
  $\frac{n}{5}$ rich red points or a convex set $C$ with the desired
  properties.

  If there exists some convex set $C$ in $\mathbb{R}^2$, such that
  $|S\cap C|=\Theta(n)$ and $\delta(S\cap C) = \Omega(\sqrt[3]{n})$,
  then the theorem is proven. Hence, assume its nonexistence.
  Let $S_i$ be the bi-colored set of points in iteration
  step~$i$, and let $R_i$ and $B_i$ be its color classes.
  Further, let $h_i$ be the number of convex hull points of $R_i$ and let
  $X_i=S_i\cap\Conv{(R_i)}$. See also Figure~\ref{fig:T22_sketch} for an
  illustration of the different sets.
  The iteration starts with $S_1=S$. For $i>1$ smaller sets $S_{i+1}$
  are constructed, by removing one red point from $S_i$.
  Considering the $i$-th iteration $(1\leq i\leq \frac{n}{5})$, we can
  state the following relations:
  \begin{enumerate}
  \item[(1)] $|S_i|=|S|-(i-1)>n-\frac{n}{5}+1=\Theta(n)$. 
  \item[(2)] $\delta(S_i)<\frac{\sqrt[3]{n}}{20}$. \\
    By relation~(1), $|S_i|=\Theta(n)$. Thus, if
    $\delta(S_i)\geq\frac{\sqrt[3]{n}}{20}$, then we can set
    $C=\Conv{(S_i)}$, implying $|S\cap C|=\Theta(n)$ and $\delta(S\cap
    C)=\Omega(\sqrt[3]{n})$, which we assumed not to exist.
  \item[(3)] $|R_i|>\frac{2n}{5}-\frac{\sqrt[3]{n}}{40}$. \\
    Using inequality~(1) and~(2), and Observation~\ref{obs:ds_bound_2D}
    we get $|R_i|>\frac{|S_i|-\frac{\sqrt[3]{n}}{20}}{2} >
    \frac{4n}{10}-\frac{\sqrt[3]{n}}{40}$.
  \item[(4)] $\delta(X_i)<\frac{\sqrt[3]{n}}{20}$. \\
    Obviously, $|X_i|\geq|R_i|=\Theta(n)$, by inequality~(3). Thus,
    $\delta(X_i)\geq\frac{\sqrt[3]{n}}{20}$ again supplies us with some
    $C=\Conv{(S_i)}$, which we assumed not to exist.
  \item[(5)] $|S_i\setminus X_i|<\frac{\sqrt[3]{n}}{10}$. \\
    Note that $|S_i\setminus X_i|=|B_i\setminus X_i|$. Using
    inequality~(4) and Observation~\ref{obs:ds_bound_2D} we get
    $|B_i\cap\Conv(X_i)|=|B_i|-|S_i\setminus X_i| = |X_i\setminus R_i|
    > |R_i| - \frac{\sqrt[3]{n}}{20}$. Using inequality~(2) we get
    $|S_i\setminus X_i| < |B_i|-|R_i| + \frac{\sqrt[3]{n}}{20} \leq
    \delta(S_i) + \frac{\sqrt[3]{n}}{20} < 2\frac{\sqrt[3]{n}}{20}$.
  \item[(6)] $h_i<\frac{\sqrt[3]{n}}{10}$. \\
    Let $X_i'=X_i\setminus (X_i\cap\CH(R_i))$. Obviously,
    $|X_i'|\geq|B_i|-|S_i\setminus X_i|$, and consequently
    $|X_i'|>\frac{4n}{10}-\frac{\sqrt[3]{n}}{40}-\frac{\sqrt[3]{n}}{10}=\Theta(n)$
    by inequality~(1) and~(2), Observation~\ref{obs:ds_bound_2D}, and
    inequality~(5).
    Therefore, we can assume $\delta(X_i') =
    \big|(|R_i|-h_i)-|X_i\setminus R_i|\big| <
    \frac{\sqrt[3]{n}}{20}$, because the contrary would imply the
    existence of some $C=\Conv{(X_i')}$, which we assumed not to
    exist.
    From $|X_i\setminus R_i|-(|R_i|-h_i) < \frac{\sqrt[3]{n}}{20}$ and
    inequality~(4) we get $h_i < |R_i| - |X_i\setminus R_i| +
    \frac{\sqrt[3]{n}}{20} = \delta(X_i)+\frac{\sqrt[3]{n}}{20} <
    \frac{\sqrt[3]{n}}{20}+\frac{\sqrt[3]{n}}{20}$.
  \end{enumerate}

  Using these inequalities we can prove the existence of rich
  points. Let $p_1,\ldots,p_{h_i}$ be the convex hull points of $R_i$
  in counter clock-wise order. Triangulate $CH(R_i)$ by adding the
  diagonals $p_1p_j$, for $3\leq j\leq (h_i-1)$. In the resulting
  triangulation let $\triangle_j$, $2\leq j\leq (h_i-1)$, be the
  triangle $p_1p_jp_{j+1}$. With $S(\triangle_j)$ denote the
  bi-colored set of points interior to $\triangle_j$ and let
  $R(\triangle_j)$ and $B(\triangle_j)$ be its color classes.

  \begin{enumerate}
  \item[(7)] $\delta(S(\triangle_j)) =
    \big||R(\triangle_j)|-|B(\triangle_j)|\big| <
    \frac{\sqrt[3]{n}}{10}$ for every $2\leq j\leq (h_i-1)$. \\
    Assume the contrary: $\delta(S(\triangle_j)) \geq
    \frac{\sqrt[3]{n}}{10}$ for some $\triangle_j$, $2\leq j\leq
    (h_i-1)$.
    Consider the three regions
    $(\triangle_{2}\cup\ldots\cup\triangle_{j-1})$, $\triangle_j$, and
    $(\triangle_{j+1}\cup\ldots\cup\triangle_{h_i-1})$. At least one of
    these three regions contains at least $\frac{|X_i|-h_i}{3} =
    \frac{|S_i|-|S_i\setminus X_i|-h_i}{3} >
    \frac{1}{3}\left(n-\frac{n}{5}+1-\frac{\sqrt[3]{n}}{10}-\frac{\sqrt[3]{n}}{10}\right)
    > \frac{n}{5}$ interior points, by inequality~(1), (5),
    and~(6).

    If $|S(\triangle_j)| \geq \frac{n}{5} = \Theta(n)$, then we can
    set $C=\Conv{(S(\triangle_j))}$, which we assumed not to
    exist. Thus assume w.l.o.g. that region
    $(\triangle_{2}\cup\ldots\cup\triangle_{j-1})$ has at least
    $\frac{n}{5}$ interior points, i.e.,
    $|S(\triangle_{2})\cup\ldots\cup S(\triangle_{j-1})| \geq
    \frac{n}{5} = \Theta(n)$.
    Note that also $|S(\triangle_{2})\cup\ldots\cup
    S(\triangle_{j-1})\cup S(\triangle_{j})| =
    |S(\triangle_{2})\cup\ldots\cup
    S(\triangle_{j-1})|+|S(\triangle_{j})| \geq \frac{n}{5} =
    \Theta(n)$.
    Then either the points inside region
    $(\triangle_{2}\cup\ldots\cup\triangle_{j-1})$ have high
    discrepancy, $\delta\left(S(\triangle_{2})\cup\ldots\cup
      S(\triangle_{j-1})\right)\geq\frac{\sqrt[3]{n}}{20}$, and thus
    we can set $C=\Conv{\left(S(\triangle_{2})\cup\ldots\cup
        S(\triangle_{j-1})\right)}$, or the discrepancy in region
    $(\triangle_{2}\cup\ldots\cup\triangle_{j-1}\cup\triangle_{j})$ is
    high, $\delta\left(S(\triangle_{2})\cup\ldots\cup
      S(\triangle_{j-1})\cup S(\triangle_j)\right)\geq
    \delta(S(\triangle_j)) -
    \delta\left(S(\triangle_{2})\cup\ldots\cup
      S(\triangle_{j-1})\right) > \frac{\sqrt[3]{n}}{20}$ and thus we
    can set $C=\Conv{\left(S(\triangle_{2})\cup\ldots\cup
        S(\triangle_{j-1})\cup S(\triangle_j)\right)}$.
    Both times the existence of a convex set $C$, with $|S\cap
    C|=\Theta(n)$ and $\delta(S\cap C)=\Omega(\sqrt[3]{n})$, is a
    contradiction to its assumed nonexistence and consequently a
    contradiction to the assumed existence of some $\triangle_j$,
    $2\leq j\leq (h_i-1)$, with $\delta(S(\triangle_j)) \geq
    \frac{\sqrt[3]{n}}{10}$.
  \end{enumerate}

  By inequalities~(3) and~(6) we have
  $\sum_{j=2}^{h_i-1}(|R(\triangle_j)|) = |R_i|-h_i >
  \frac{4n}{10}-\frac{\sqrt[3]{n}}{40}-\frac{\sqrt[3]{n}}{10} >
  \frac{3n}{10}$. Hence, there exists a $\triangle_j$, such that
  $|R(\triangle_j)| > \frac{3n}{10(h_i-2)} > 3n^{2/3}$. Further, using
  inequality~(7) and Observation~\ref{obs:ds_bound_2D} we have
  $|B(\triangle_j)| < |R(\triangle_j)| + \frac{\sqrt[3]{n}}{10}$.
  Applying Lemma~\ref{lem:ordergen} for $d=2$ we know that there exist
  at least $|R(\triangle_j)| + \sqrt{|R(\triangle_j)|} +1$ interior
  disjoint red triangles, each with a point in $\CH(\triangle_j)$. At
  least $|R(\triangle_j)|+\sqrt{|R(\triangle_j)|}+1-|B(\triangle_j)| >
  |R(\triangle_j)|+\sqrt{|R(\triangle_j)|}+1-|R(\triangle_j)|-\frac{\sqrt[3]{n}}{10}
  > \sqrt{3n^{2/3}} - \frac{\sqrt[3]{n}}{10} > \sqrt[3]{n}$ of these
  triangles are empty of points, and at least a third of them has the
  same point $p$ in $\CH(\triangle_j)$ and thus in $\CH(R_i)$. Hence,
  $p$ is a rich point.

  \bigskip If $i<\frac{n}{5}$ then let $S_{i+1}=S_i\setminus\{p\}$,
  $i=i+1$, and iterate. As all triangles counted so far have $p$ as a
  vertex, and $p$ does not belong to the point sets of future
  iterations, we do not overcount.
  The process either terminates with a convex set $C$, such that
  $|S\cap C|=\Theta(n)$ and $\delta(S\cap C) = \Omega(\sqrt[3]{n})$,
  or it ends after $\frac{n}{5}$ steps.
  For each rich point we can count at least $\frac{\sqrt[3]{n}}{3}$
  empty red triangles. As we get $\frac{n}{5}$ rich points and do not
  overcount we get $\frac{n}{5}\cdot\frac{\sqrt[3]{n}}{3} =
  \Omega(n^{4/3})$ empty red triangles in total.
\end{proof}

Combining Theorem~\ref{thm:newbasecase2d} with
Lemma~\ref{lem:discrepancyd2k2} proves the bound of $\Omega(n^{4/3})$
empty monochromatic triangles for the 2-colored case in the plane,
already shown in~\cite{pachmono}.
However, Theorem~\ref{thm:newbasecase2d} can be generalized to
$\mathbb{R}^d$:

\begin{theorem}\label{thm:newtriangordisgen}
  Let $S$ be a bi-colored set of $n$ points in general position in
  $\mathbb{R}^d$ $(d\geq2)$, partitioned into a red point set~$R$ and
  a blue point set~$B$.
  Then either there exist $\Omega(n^{d-2/3})$ empty red
  \mbox{$d$-simplices}, or there exists a convex set $C$ in
  $\mathbb{R}^d$, such that $|S\cap C|=\Theta(n)$ and $\delta(S\cap C)
  = \Omega(\sqrt[3]{n})$.
\end{theorem}

\begin{proof}
  We prove the theorem by induction on the dimension $d$ (recall that
  $d$ is a constant, independent of $n$), and use
  Theorem~\ref{thm:newbasecase2d} as an induction base for $d=2$.
  Consider the induction step $(d-1)\longrightarrow d$, for
  $d>2$.
  If $\delta(S)\geq \sqrt[3]{n}$ then $C=\Conv(S)$ is the desired
  convex set, with $|S\cap C|=\Theta(n)$ and $\delta(S\cap C) =
  \Omega(\sqrt[3]{n})$.
  Thus assume that $\delta(S) < \sqrt[3]{n}$. From
  Observation~\ref{obs:ds_bound_2D} we know that
  $|R|>\frac{n-\sqrt[3]{n}}{2} = \Theta(n)$.

  Let $p\in R$ be a red point.
  For every point $q \in S\setminus\{p\}$ let $r_{q}$ be the infinite
  ray with origin $p$ and passing through $q$.
  Let $\Pi'$ and $\Pi''$ be two $(d-1)$-dimensional hyperplanes
  containing $\Conv(S)$ between them and not parallel to any of the
  rays $r_q$. See Figure~\ref{fig:T23_sketch} on
  page~\pageref{fig:T23_sketch} (for the very similar proof of
  Theorem~\ref{thm:triang_or_dis_d}) for a sketch.
  Project from $p$ every point in \mbox{$S\!\setminus\!\!\{p\}$} to
  $\Pi'$ or $\Pi''$, in the following way. Every ray $r_q$ intersects
  either $\Pi'$ or $\Pi''$ in a point $q'$ or $q''$,
  respectively. Take $q'$ or $q''$ to be the
  projection of $q$ from $p$. Let $S'$ and $S''$ be the sets of these projected
  points in $\Pi'$ and $\Pi''$, respectively.
  The bigger set, assume w.l.o.g. $S'$ in $\Pi'$, is a set of at least
  $\frac{n-1}{2}$ points in general position in $\mathbb{R}^{d-1}$. 

  Apply the induction hypothesis to $S'$ and get either (a)
  $\Omega(n^{d-1-2/3})$ empty red \mbox{$(d\!-\!1)$}-simplices, or (b)
  a convex set $C$ in $\mathbb{R}^{d-1}$, such that $|S'\cap
  C|=\Theta(n)$ and $\delta(S'\cap C) = \Omega(\sqrt[3]{n})$.

  For case~(b) observe, that the preimage of a point set of a convex
  set in $\Pi'$ is the point set of a convex set in
  $\mathbb{R}^{d}$. Hence, $C$ is a convex set in $\mathbb{R}^{d}$,
  such that $|S\cap C|=\Theta(n)$ and $\delta(S\cap C) =
  \Omega(\sqrt[3]{n})$.

  For case~(a) note that, if $X$ is the vertex set of an empty red
  \mbox{$(d\!-\!1)$}-simplex in $\Pi'$, then $\Conv(X\cup{p})$ is an
  empty red $d$-simplex in $\mathbb{R}^{d}$.
  Repeat the projection and the induction for each red point $p\in R$
  and assume that this always results in case~(a) (because the proof
  is completed if case~(b) happens once).
  This results in a total of $\frac{|R|}{d+1}\cdot\Omega(n^{d-1-2/3})
  = \Omega(n^{d-2/3})$ empty red $d$-simplices, as each $d$-simplex
  gets overcounted at most $(d+1)$ times.
\end{proof}

Combining Theorem~\ref{thm:newtriangordisgen} with the two variants of
the ``Discrepancy Lemma'' for the bi-colored case
(Lemmas~\ref{lem:discrepancyd2k2} and~\ref{lem:discr2colors}), allows
us to generalize the bound on the number of empty monochromatic
triangles for the bi-colored case in the plane, to $\mathbb{R}^d$.

\begin{theorem}\label{thm:ems_in_d2+k2}
  Any bi-colored set $S$ of $n$ points in general position in
  $\mathbb{R}^d$, $d \geq 2$, determines
  $\Omega(n^{d-2/3})$ empty monochromatic $d$-simplices.
\end{theorem}

\begin{proof}
  By Theorem~\ref{thm:newtriangordisgen} either there exist
  $\Omega(n^{d-2/3})$ empty monochromatic $d$-simplices, or there
  exists a convex set $C$ in $\mathbb{R}^d$, such that $|S\cap
  C|=\Theta(n)$ and $\delta(S\cap C) = \Omega(\sqrt[3]{n})$.

  In the former case the theorem is proven.
  In the latter case, if $d=2$ then there exist
  $\Omega(n^{2-1+1/3})=\Omega(n^{4/3})$ empty monochromatic triangles
  ($2$-simplices) by applying Lemma~\ref{lem:discrepancyd2k2} to
  $(S\cap C)$, and if $d>2$ then there exist
  $\Omega(n^{d-1+1/3})=\Omega(n^{d-2/3})$ empty monochromatic
  $d$-simplices by applying Lemma~\ref{lem:discr2colors} to $(S\cap
  C)$.
\end{proof}

\section{Conclusions} \label{sec:conclusions}

In this paper we generalized known bounds on the number of empty
monochromatic triangles and tetrahedra on colored point sets to higher
dimensions.
Our results are summarized in Table~\ref{tab:conclusion_k_over_d}
(Section~\ref{sec:intro}).

As main results, in Theorem~\ref{thm:ems_in_d2+k2}, we proved that any
bi-colored point sets in $\mathbb{R}^d$ determines $\Omega(n^{d-2/3})$
empty monochromatic $d$-simplices.
For $3 \leq k \leq n$, in Theorem~\ref{thm:ems_in_d3+}, we proved that
any $k$-colored point set in $\mathbb{R}^d$ determines
$\Omega(n^{d-k+1-2^{-d}})$ empty monochromatic $d$-simplices.
Further, we extended the linear lower bound for the number of empty
monochromatic tetrahedra in 4-colored point sets in $\mathbb{R}^3$ to
a linear lower bound for the number of empty monochromatic
$d$-simplices in \mbox{$(d+1)$-colored} point sets in $\mathbb{R}^d$,
Corollary~\ref{cor:(d+1)}.

\smallskip

In order to prove our lower bounds on the number of empty
monochromatic $d$-simplices, we proved a result that is interesting on
its own right.  Theorem~\ref{thm:dn+log} shows that a simplicial
complex with at least
$dn+\max{\left\{h,\frac{\log_2(n)}{2d}\right\}}-c_d$, with
\mbox{$c_d\!=\!d^3\!+\!d^2\!+\!d$}, $d$-simplices exists for any point
set in $\mathbb{R}^d$ (points in general position).

Although still linear, this is a first non-trivial bound, of interest
in view of the following open problem stated by Brass et
al.~\cite{bmj-rpdg-05}:
What is the maximum number $R_d(n)$ such that every set of $n$ points
in general position in $d$-dimensional space has a triangulation
consisting of at least $R_d(n)$ simplices?
Moreover, Urrutia~\cite{jorge3d} posed the following open problem:

\begin{problem}\label{pro:tetra}
  Is it true that for any point set in $\mathbb{R}^3$ there exists a
  triangulation with super linear many $3$-simplices?
\end{problem}

A positive answer to this question implies that any $k$ coloring of a
set of points with $n$ elements, always contains an empty monochromatic
simplex, $k$ constant, and $n$ sufficiently large.

Unfortunately, proving or disproving Problem~\ref{pro:tetra} seems to
be illusive and remains open. On the other hand, it is well known that
any set of $n$ points on the momentum curve $(x, x^2, x^3)$ has a
triangulation with a quadratic number of $3$-simplices. Aside from
this, we are not aware of many families of point sets in general
position in $\mathbb{R}^3$ for which it is known that there exist
triangulations with a quadratic number of $3$-simplices.

We close our paper with the following result that somehow suggests
that any point set with $n$ elements in $\mathbb{R}^d$ is not
\emph{far} from a point set that generates a quadratic number of
interior disjoint $3$-simplices:

\begin{theorem}\label{thm:quadratic}
  Any set $X$ of $n$ points in general position in $\mathbb{R}^3$ is
  contained in a set $S$ with $2n$ points in general position in
  $\mathbb{R}^3$ such that $S$ determines at least $n \choose 2$
  interior disjoint $3$-simplices.
\end{theorem}

\begin{proof}
  Let $v_1, \ldots, v_n$ be $n$ different unit vectors, no two of
  which are parallel to each other, nor parallel to any segment
  determined by any two elements of $X$. For each point $p_i$, $1\leq
  i\leq n$, in $X$ let $q_i = p_i + \varepsilon\cdot v_i$ be a point
  of \mbox{$S\!\setminus\!X$}, where $\varepsilon$ is a small enough
  constant.
  Let $\sigma_{i,j}=\Conv(\{p_i,q_i,p_j,q_j\})$ be a $3$-simplex.
  As $X$ is in general position, it is easy to see, that for all
  $(i,j)\neq(r,s)$ the $3$-simplices $\sigma_{i,j}$ and $\sigma_{r,s}$
  have disjoint interior.
\end{proof}

Note that in many instances it will not be possible to complete the
set of $3$-simplices obtained in the proof above to a full
triangulation of $S$. Nevertheless, this shows that the families of
point sets admitting a quadratic number of interior disjoint empty
$3$-simplices may not have any special properties that would allow us
to characterize them.

Clearly, the construction used in Theorem~\ref{thm:quadratic} can be
generalized to higher dimensions. We conjecture:

\begin{conjecture}\label{conj:manycolor}
  For each $d\geq3$ and every constant $k$, there exists a constant
  $f(d,k)$ such that every set $S \in \mathbb{R}^d$ of more than
  $f(d,k)$ points with arbitrary $k$-coloring has a monochromatic
  empty $d$-simplex.
\end{conjecture}

Even more so, we believe that the answer to Problem~\ref{pro:tetra} is
``Yes'' and that this actually extends to higher dimensions.

\begin{conjecture}\label{conj:superlinear}
  For each $d\geq3$ and every point set $S \in \mathbb{R}^d$ there
  exists a triangulation with super linear many $d$-simplices.
\end{conjecture}

Of course, proving this stronger Conjecture~\ref{conj:superlinear}
would imply a proof for Conjecture~\ref{conj:manycolor}: Construct a
triangulation of super linear size on the biggest color class
$R\subseteq S$. There exist only linear many differently colored
points in \mbox{$S\!\setminus\! R$} to fill the super linear many
monochromatic $d$-simplices on $R$. Hence, there exists at least one
monochromatic empty $d$-simplex.

\section*{Acknowledgements}

We want to thank David Flores-Pe\~naloza for valuable comments.

Research of Oswin Aichholzer supported by the ESF EUROCORES programme
EuroGIGA -- CRP `ComPoSe', Austrian Science Fund (FWF): I648-N18.
Research of Ruy Fabila-Monroy partially supported by CONACyT (Mexico)
grant 153984.
Research of Thomas Hackl supported by the Austrian Science Fund
(FWF): P23629-N18 `Combinatorial Problems on Geometric Graphs'.
Research of Clemens Huemer partially supported by projects
MTM2009-07242, Gen. Cat. DGR 2009SGR1040, and ESF EUROCORES programme
EuroGIGA -- CRP `ComPoSe', MICINN Project EUI-EURC-2011-4306.
Research of Jorge Urrutia partially supported by CONACyT (Mexico)
grant CB-2007/80268.

\bibliographystyle{plain}
\bibliography{afhhu-ems}

\end{document}